\documentclass[11pt]{article}
\usepackage{amsmath}
\usepackage{amsfonts}
\usepackage{color}
\usepackage{latexsym}
\usepackage{tablefootnote}
\usepackage{epsfig}
\usepackage{multirow}
\usepackage{float}
\usepackage{array}
\usepackage{amssymb,amsthm}

\newcommand*{\QEDA}{\hfill\hbox{\vrule width1.0ex height1.0ex}}

\usepackage{makecell,booktabs}
\usepackage[version=4]{mhchem}   
\usepackage[strict]{changepage}

\usepackage{amsmath}
\usepackage{amsfonts}
\usepackage{latexsym}
\usepackage{amssymb}

\newtheorem{thm}{Theorem}[section]
\newtheorem{theorem}[thm]{Theorem}

\newtheorem{lemma}[thm]{Lemma}

\newtheorem{proposition}[thm]{Proposition}

\newtheorem{corollary}[thm]{Corollary}

\newtheorem{remark}[thm]{Remark}

\newcommand{\beq}{\begin{equation}}
\newcommand{\eeq}{\end{equation}}
\newcommand{\beqa}{\begin{eqnarray}}
\newcommand{\eeqa}{\end{eqnarray}}
\newcommand{\beqas}{\begin{eqnarray*}}
\newcommand{\eeqas}{\end{eqnarray*}}
\newcommand{\bi}{\begin{itemize}}
\newcommand{\ei}{\end{itemize}}

\newcommand{\nn}{\nonumber}

\newcommand{\R}{\mathbb{R}}

\newcommand{\lam}{{\lambda}}

\newcommand{\inner}[2]{\langle #1,#2\rangle}

\newcommand{\argmin}{\mathrm{argmin}\,}

\newcommand{\dom}{\mathrm{dom}\,}

\newcommand{\Argmin}{\mathrm{Argmin}\,}

\newcommand{\mConv}[1]{\overline{\mbox{\rm Conv}}_\mu\,(\R^{#1})}

\newcommand{\bConv}[1]{\overline{\mbox{\rm Conv}}\,(\R^{#1})}

\newcommand{\barco}{\overline{\mbox{co}}\,}
\newcommand{\tx}{\tilde x}

\begin{document}
	\title{A unified analysis of
	a class of proximal bundle methods \\
	for solving hybrid convex composite optimization problems}
	\date{October 3, 2021 (first revision: January 7, 2023; second revision: March 27, 2023)}
    \author{
		Jiaming Liang \thanks{Department of Computer Science, Yale University, New Haven, CT 06511 (email: {\tt jiaming.liang@yale.edu}).
		}
  \qquad
		Renato D.C. Monteiro \thanks{School of Industrial and Systems
			Engineering, Georgia Institute of
			Technology, Atlanta, GA 30332
			(email: {\tt renato.monteiro@isye.gatech.edu}). This work
			was partially supported by ONR Grant N00014-18-1-2077 and AFOSR Grant FA9550-22-1-0088.}
	}
	\maketitle
	
	\begin{abstract}
		
		This paper presents a  proximal bundle (PB) framework based on a generic bundle update scheme for
		solving the hybrid convex composite optimization (HCCO) problem and
		establishes
		a common iteration-complexity bound for any variant belonging to it.
		As a consequence,
	iteration-complexity bounds
		for three
		PB variants based on different
		bundle update schemes are obtained in
		the HCCO context for the first time and in a unified manner.
		While two of the
		PB variants are universal
		(i.e., their implementations do not require parameters associated with the HCCO instance),
		the other newly (as far as the authors are aware of) proposed one
		is not  but has the
		advantage that it generates simple, namely one-cut,
		bundle models.
	    The paper also presents a universal adaptive PB variant
	    (which is not necessarily an instance of the framework) based on one-cut models   and
	    shows that its iteration-complexity is the same as the two aforementioned universal PB variants.\\

		\\
		
	\end{abstract}
	 
	\section{Introduction}\label{sec:intro}
	
	
	
	
	
	Let $ f, h: \R^{n} \rightarrow \R\cup \{ +\infty \} $ be proper lower semi-continuous  convex functions such that
	$ \dom h \subseteq \dom f $ and
	$h - \mu \|\cdot\|^2/2$ is convex
	for some $ \mu\ge 0 $, and
	consider
	the optimization problem
	\begin{equation}\label{eq:ProbIntro}
	\phi_{*}:=\min \left\{\phi(x):=f(x)+h(x): x \in \R^n\right\}.
	\end{equation}
	It is said that \eqref{eq:ProbIntro} is a hybrid convex composite optimization (HCCO) problem if there exist
	nonnegative scalars $M_f$ and $L_f$
	and a
	first-order oracle $f':\dom h \to \R^n$
	(i.e., $f'(x)\in \partial f(x)$
	for every $x \in \dom h$)
	satisfying the $(M_f,L_f)$-hybrid condition,
	namely:
	$\|f'(u)-f'(v)\| \le 2M_f + L_f \|u-v\|$ for every $u,v \in \dom h$. The main goal of this paper
	is to study the complexity of
	proximal bundle methods for solving the HCCO problem \eqref{eq:ProbIntro}
	based on different bundle update schemes.
	Instead of focusing on a particular proximal bundle method, our unified approach considers
	a framework of
	 generic proximal bundle
	methods
 (referred to as the
	GPB framework)
 based on
	a generic bundle update scheme,
	and establishes a common iteration-complexity
	bound  for all instances
	belonging to it.
	
	{\it Method outline.}
	Like all other proximal bundle methods, an iteration of a GPB variant solves the prox bundle subproblem
	\begin{equation}\label{eq:x-pre}
	    x = \underset{u\in \R^n}\argmin \left \{\Gamma (u) + \frac{1}{2\lam} \|u-x^c\|^2 \right\}
	\end{equation}
	where $\lam$ is the prox stepsize, and
	$x^c$ and $\Gamma$ are the current
	prox-center and bundle function, respectively.
	Moreover, it also performs two types of iterations, i.e., serious and null ones.
    In a serious iteration, the prox-center is
    updated to $x^c \leftarrow x $ and
    the updated bundle function $\Gamma^+$
    is chosen so as to
    satisfy $\Gamma^+\ge \ell_f(\cdot;x)+h$
    where $\ell_f(\cdot;x)= f(x)+\inner{f'(x)}{\cdot-x}$.
    In a null iteration, the prox-center does not
    change but $\Gamma$ is updated
    according to a certain
    bundle update scheme (which is usually more restrictive than the ones in the serious iterations).
    
	In order to illustrate the use of the
	generic bundle update scheme, this paper
	considers three specific well-known
	bundle update schemes and shows that
	they can all be viewed as special cases
	of the generic one. We now briefly describe
	the specific ones in the next three itemized paragraphs.
	\begin{itemize}
	    \item [(E1)] {\bf one-cut scheme}:
	    This scheme obtains $\Gamma^+$ as
	    \begin{equation}\label{eq:affine}
        \Gamma^+= \Gamma^+_\tau := \tau \Gamma + (1-\tau) [\ell_f(\cdot;x)+h]
        \end{equation}
        where $x$ is as in \eqref{eq:x-pre} and $\tau\in (0,1)$ depends on $(L_f,M_f,\mu)$. Clearly, if
        $\Gamma$ is the sum of $h$ and
        an affine function underneath $f$,
        then so is $\Gamma^+$.


	    
	   
	    \item [(E2)] \textbf{two-cuts scheme:}
	    Assume that
	    $\Gamma=\max \{A_f,\ell_f(\cdot;x^-)\}+h$ where $A_f$ is an affine function satisfying $A_f\le f$ and $x^-$ is the previous iterate.
	    This scheme sets the next bundle function $\Gamma^+$ to one similar to
	    $\Gamma$ but with  $(x^-,A_f)$ replaced by
	    $(x,A_f^+)$ where
	    $A_f^+ =  \theta A_f + (1-\theta) \ell_f(\cdot;x^-)$ for some
	    $\theta\in[0,1]$ which does not
	    depend on $(L_f,M_f,\mu)$.
	    
	    \item [(E3)] \textbf{multiple-cuts scheme:}
	    The current bundle function $\Gamma$ is of the form $\Gamma=\Gamma(\cdot;B)$ where
	    $B \subset \R^n$ is a finite set (i.e., the current  bundle set) and $\Gamma(\cdot;B)$ is defined as
	    \begin{equation}\label{eq:Gamma-E2}
	        \Gamma(\cdot;B) := \max \{ \ell_f(\cdot;b) : b \in B \}+h.
	    \end{equation}
	    This scheme obtains $\Gamma^+$ as $\Gamma^+ = \Gamma(\cdot;B^+)$ where
	    $B^+$ is the updated bundle set obtained
	    by possibly 
	    removing some points from $B$ and
	    then adding the most recent $x$ to
	    the resulting set.
	\end{itemize}
	Throughout out the paper, we refer to the GPB instances based on (E1), (E2) and (E3) as  1C-PB, 2C-PB and MC-PB, respectively.
	

	
	{\it Contribution.}
	Regardless of the parameter
	triple $(L_f, M_f, \mu)$,
	it is shown that
	the iteration-complexity for any GPB variant to obtain a $\bar \varepsilon$-solution of the HCCO problem \eqref{eq:ProbIntro} (i.e., a point $\bar x\in \dom h$ satisfying $\phi(\bar x)-\phi^*\le \bar \varepsilon$) is
	\begin{equation}\label{eq:bound-mu}
		{\cal O}\left(\min \left\lbrace \frac{ (M_f^2+\bar \varepsilon L_f) d_0^2}{\bar \varepsilon^2}, \left( \frac{M_f^2+\bar \varepsilon L_f}{ \mu \bar \varepsilon} + 1 \right) \log\left( \frac{ \mu d_0^2}{ \bar \varepsilon} + 1\right) \right\rbrace  +1 \right)
	\end{equation}
	for a large range of prox stepsizes $\lam$, where 
  $d_0$ denotes the distance of the initial point $x_0$ to the optimal solution set of \eqref{eq:ProbIntro}.
	Since 2C-PB and MC-PB methods do not rely on $(L_f, M_f)$, a sharper
	iteration-complexity bound can be obtained 
	for them by replacing
	$(M_f,L_f)$ in \eqref{eq:bound-mu} by $(\bar M_f, \bar L_f)$, respectively, where
	$(\bar M_f, \bar L_f)$ is the unique pair
	which minimizes $M_f^2 +\bar \varepsilon L_f$ over the set of
	pairs $(M_f,L_f)$ satisfying the $(M_f,L_f)$-hybrid condition of $f'$.
	Moreover,
	even though this sharper complexity bound can not be shown for 1C-PB, Section \ref{sec:adap} presents an
	adaptive version of
	this variant where $\tau$
	in \eqref{eq:affine},
	instead of being chosen
	as a function of $(L_f,M_f)$,
	is adaptively searched
	so as to satisfy a key inequality
	condition. Finally,
	Section \ref{sec:adap} also shows that this adaptive variant
	 has the same iteration-complexity as that of 2C-PB and MC-PB.

    {\it Related literature.}
    Proximal bundle methods are known to be efficient algorithms for solving nonsmooth convex composite optimization (NCCO) problems, i.e., instances of \eqref{eq:ProbIntro} for which there exists
    $M_f \ge 0$ such that 
    the hybrid condition holds with $L_f=0$.
    Some preliminary ideas towards the development of
the proximal bundle method were first presented in \cite{lemarechal1975extension,wolfe1975method}
and formal presentations of the method were given in
\cite{lemarechal1978nonsmooth,mifflin1982modification}.
Convergence analysis of the proximal bundle method for
NCCO problems has been broadly discussed
in the literature and can be found for example
in the textbooks \cite{ruszczynski2011nonlinear,urruty1996convex}.
Different bundle management policies in the context of proximal bundle methods are discussed for example in
\cite{du2017rate,frangioni2002generalized,kiwiel2000efficiency,de2014convex,ruszczynski2011nonlinear,van2017probabilistic}.








Iteration-complexity bounds have been established for some proximal bundle methods in the context of the NCCO problem with $\mu=0$
(see for example \cite{astorino2013nonmonotone,diaz2021optimal,kiwiel2000efficiency,liang2020proximal}).
Papers \cite{astorino2013nonmonotone,kiwiel2000efficiency} both consider the NCCO problem where
	$h$ is the indicator function
	of a nonempty closed convex set, and \cite{diaz2021optimal} considers the NCCO problem where $h$ is identically zero.
	Moreover, 
paper \cite{kiwiel2000efficiency}
	obtains the first ${\cal O} (\bar \varepsilon^{-3})$ complexity bound, and \cite{astorino2013nonmonotone,diaz2021optimal} subsequently also derive an ${\cal O}(\bar\varepsilon^{-3})$ bound. 
	On the other hand, a previous authors' paper \cite{liang2020proximal} proposes a proximal bundle variant using a novel condition to decide
  whether to perform a serious or null iteration
  which does not necessarily yield a function value decrease.
	More importantly, \cite{liang2020proximal} establishes the first ${\cal O} (\bar \varepsilon^{-2})$ complexity bound for a large range of prox stepsizes, and shows that the bound is indeed optimal.

	
	More specialized
	iteration-complexity bounds for some proximal bundle methods in the context of
	the NCCO problem
	with $\mu>0$ have also been
	established in \cite{diaz2021optimal,du2017rate,liang2020proximal}.
	More specifically, 
     \cite{du2017rate}
	derives a $\tilde {\cal O} (\bar \varepsilon^{-1})$ iteration-complexity bound for a proximal bundle method with prox stepsize set to
 $\lam=1/\mu$. Moreover,
 improving on the analysis of
 \cite{du2017rate},
 paper 
    \cite{diaz2021optimal}
 establishes
 the optimal bound ${\cal O} (\bar \varepsilon^{-1})$ for the same method.
    Finally, \cite{liang2020proximal} also establishes a
	$\tilde {\cal O} (\bar \varepsilon^{-1})$ iteration-complexity bound for its
	proximal bundle variant.
    In contrast to \cite{diaz2021optimal,du2017rate}, the bound in \cite{liang2020proximal} is
    shown to be optimal (up to a logarithmic term) 
 for a large range of prox stepsizes.

    	The current paper improves \cite{liang2020proximal} in the following aspects:
1) it deals with the more general HCCO problem;
2) in contrast to \cite{liang2020proximal},
it  nowhere assumes
that $h$ is Lipschitz continuous nor imposes any condition
on the parameter $\mu$, and
shows that the iteration-complexity bound
\eqref{eq:bound-mu} holds for prox stepsize ranges
which are larger than the ones in \cite{liang2020proximal};
3) while the proximal bundle variant of \cite{liang2020proximal} is based on
the bundle
update scheme (E3), GPB 
is a framework
based on a generic bundle update scheme which
contains
proximal bundle variants
based on different update schemes
(such as (E1)-(E3));
moreover, its unified analysis presented
here applies
to all these proximal bundle variants;
and 4) as far as the authors are aware of, it presents and analyzes for the first time
a one-cut proximal bundle method
for both NCCO and HCCO problems and also presents a universal variant of such method.


    Another method related, and developed subsequently, to
the proximal bundle method is the  bundle-level method,
which was first proposed in \cite{lemarechal1995new}
and extended in many ways in \cite{ben2005non,kiwiel1995proximal,lan2015bundle}.
These methods have been shown to have optimal
iteration-complexity in the setting
of the NCCO problem with $h$ being the indicator function
of a compact convex set. 
Since their generated subproblems
do not have a proximal term, and hence do not use a prox stepsize,
they are different
from the ones studied in this paper.
    Finally, paper \cite{de2016doubly} presents a doubly stabilized bundle method for solving NCCO problems whose prox subproblems combine elements from both proximal bundle and bundle-level methods
    and analyzes its
    asymptotic convergence (but not its iteration-complexity).

    {\it Organization of the paper.}
    Subsection~\ref{subsec:DefNot}  presents basic definitions and notation used throughout the paper.
Section~\ref{sec:GPB} formally
describes the assumptions
on the HCCO problem \eqref{eq:ProbIntro},
reviews
the constant stepsize composite subgradient (CS-CS) method
and discusses its iteration-complexity.
Section~\ref{sec:framework} presents a generic bundle update scheme, describes the GPB framework and states the main results of the paper, namely, the iteration-complexity of GPB.
Section~\ref{sec:analysis} contains three subsections, and they provide the analysis of bounds on the number of the serious, null and total iterates, respectively.
Section \ref{sec:adap} presents the adaptive variant of 1C-PB and establishes the iteration-complexity of it.
Section~\ref{sec:conclusion} presents some concluding remarks and possible extensions.
Appendix~\ref{sec:technical} provides a few useful technical results.
Appendix~\ref{sec:recursive} presents two recursive formulas and their related results.
Appendix~\ref{sec:CS pf} provides the proof of 
the iteration-complexity for the CS-CS method, and describes an adaptive variant of CS-CS and establishes its iteration-complexity.
Finally, Appendix~\ref{sec:pf E2E3} provides the proofs of properties of bundle update schemes (E2) and (E3).

    \subsection{Basic definitions and notation} \label{subsec:DefNot}
    
    Let $\R$ denote the set of real numbers.
    Let $ \R_+ $ and $ \R_{++} $ denote the set of non-negative real numbers and the set of positive real numbers, respectively.
	Let $\R^n$ denote the standard $n$-dimensional Euclidean space equipped with  inner product and norm denoted by $\left\langle \cdot,\cdot\right\rangle $
	and $\|\cdot\|$, respectively. 
	Let $\log(\cdot)$ denote the natural logarithm.

	Let $\Psi: \R^n\rightarrow (-\infty,+\infty]$ be given. Let $\dom \Psi:=\{x \in \R^n: \Psi (x) <\infty\}$ denote the effective domain of $\Psi$ 
	and $\Psi$ is proper if $\dom \Psi \ne \emptyset$.
	A proper function $\Psi: \R^n\rightarrow (-\infty,+\infty]$ is $\mu$-convex for some $\mu \ge 0$ if
	$$
	\Psi(\alpha z+(1-\alpha) z')\leq \alpha \Psi(z)+(1-\alpha)\Psi(z') - \frac{\alpha(1-\alpha) \mu}{2}\|z-z'\|^2
	$$
	for every $z, z' \in \dom \Psi$ and $\alpha \in [0,1]$.
 The set of all proper lower semicontinuous $\mu$-convex functions is denoted by $\mConv{n}$.
 When $\mu=0$, we simply denote
    $\mConv{n}$ by $\bConv{n}$.
	For $\varepsilon \ge 0$, the \emph{$\varepsilon$-subdifferential} of $ \Psi $ at $z \in \dom \Psi$ is denoted by
	\[
        \partial_\varepsilon \Psi (z):=\left\{ s \in\R^n: \Psi(z')\geq \Psi(z)+\left\langle s,z'-z\right\rangle -\varepsilon, \forall z'\in\R^n\right\}.
        \]
	The subdifferential of $\Psi$ at $z \in \dom \Psi$, denoted by $\partial \Psi (z)$, is by definition the set  $\partial_0 \Psi(z)$.

    Finally, even though ${\cal O}(\cdot)$ is a well-known concept in the study of complexity of algorithms, it is convenient for the purpose of our presentation to give a slightly stronger meaning to it,
    namely, if $f$ and $g$ are two positive functions defined in a certain set $\Omega$, the notation
    $f(x) = {\cal O}(g(x))$ means that
    there exists constant $C>0$ such that
    $f(x) \le C g(x)$ for all $x \in \Omega$.
    
	
	\section{Problem of interest and a review of the CS-CS method}\label{sec:GPB}

	This section consists of two subsections. The first one  describes the main problem and the assumptions imposed on it. The second one reviews the CS-CS method and an adaptive variant of it, and describes their iteration-complexity bounds for obtaining a
	$\bar \varepsilon$-solution of the main problem.
	
	\subsection{Main problem and assumptions}
	
	The problem of interest in this paper is \eqref{eq:ProbIntro} which is
	assumed to satisfy the following conditions
	for some triple
	$(L_f, M_f,\mu) \in \R_+^3 $:
	\begin{itemize}
		\item[(A1)]
		$f \in \bConv{n}$ and $h \in \mConv{n}$ are such that
		$\dom h \subset \dom f$, and a subgradient oracle,
		 i.e.,
		a function $f':\dom h \to \R^n$
		satisfying $f'(x) \in \partial f(x)$ for every $x \in \dom h$, is available;
		\item[(A2)]
		the set of optimal solutions $X^*$ of
		problem \eqref{eq:ProbIntro} is nonempty;
		\item[(A3)]
		for every $x,y \in \dom h$,
		\[
		\|f'(x)-f'(y)\| \le 2M_f + L_f \|x-y\|.
		\]
	\end{itemize}

	Throughout this paper, an instance of \eqref{eq:ProbIntro} means a triple $(f,f';h)$ satisfying
conditions (A1)-(A3) for some triple of parameters $(L_f,M_f,\mu) \in \R_+^3$.

We now add a few remarks about assumptions (A1)-(A3).
First,
 letting 
    \begin{equation}\label{def:gamma}
	\ell_f(\cdot;x) := f(x)+\inner{f'(x)}{\cdot-x} \quad \forall x\in \dom h,
	\end{equation}
    then it is well-known that (A3) implies that for every $x,y \in \dom h$,
	\begin{equation}\label{ineq:est}
	f(x)-\ell_f(x;y) \le 2M_f \|x-y\| + \frac{L_f}{2}\|x-y\|^2.
	\end{equation}
 Second,
 an obvious example of $f$ satisfying (A3) is the sum of an $M_f$-Lipschitz continuous function and a function whose gradient is $L_f$-Lipschitz continuous, e.g.,
$f(x)=M_f\|x\|+L_f\|x\|^2/2$.
 Third, another way of obtaining
 functions $f$ satisfying (A3)
 is discussed in
 Proposition \ref{lem:hybrid} below.

 We now discuss other quantities which, in addition to the parameters $L_f$, $M_f$, and $\mu$, are also
used in the complexity
bounds obtained in 
this paper.
For a given initial
	point $x_0 \in \dom h$,
 we denote its distance to
	 $X^*$ as
\begin{equation}\label{def:d0}
	d_0 := \|x_0-x_0^*\|, \ \ \mbox{\rm where} \ \ \ 
	x_0^* := \argmin \{\|x_0-x^*\|: x^*\in X^*\}.
\end{equation}
Alternative quantities which are used in place of
$M_f$ and $L_f$ are as follows.
	First note that
 the set $\Omega \subset \R^2_+$ consisting of
	 the pairs $(M_f,L_f)$
	satisfying (A3) is easily seen to be a (nonempty) closed convex set. Moreover,
	for a given tolerance
	$\bar \varepsilon>0$, it is easily seen that
	there exists a unique
	pair $(\bar M_f(\bar \varepsilon),\bar L_f(\bar \varepsilon))$ which minimizes
	$M_f^2+\bar \varepsilon L_f$ over $\Omega$ and,
	without any loss of clarity,
	we denote this pair simply by
	$(\bar M_f,\bar L_f)$ and
	define
		\begin{equation}\label{def:T}
	    T_{\bar \varepsilon} := \left( \bar M_f^2+\bar \varepsilon \bar L_f \right)^{1/2}.
	\end{equation}
	Moreover, if there exists a
	pair $(M_f,0)$ satisfying
	(A3), then
	the smallest $M_f$
	with this property is denoted by $\bar M_{f,0}$; otherwise,
	if no such pair exists, then
	we set $\bar M_{f,0}:=\infty$.
	Finally, it is easily seen that
	$\bar M_{f,0} \ge T_{\bar \varepsilon} \ge \bar M_f$
	and that any one of these two inequalities can hold strictly. For example,
	if $f=\|\cdot\| + \|\cdot\|^2/2$ and $h \equiv 0$, then we can easily see that $\bar M_{f,0}=\infty$,
	$\bar M_f=1$, and $T_{\bar \varepsilon} \in (1,\infty)$ for any $\bar \varepsilon>0$.


The following result, whose proof is postponed to Appendix~\ref{sec:technical},
gives conditions on $(f,h)$ which guarantee that (A3) holds.

\begin{proposition}\label{lem:hybrid}
Assume that (A1) holds and that, for some $\nu \in (0,1) $, the function $f'$ in
(A1) satisfies
    \begin{equation}\label{eq:hybrid-holder}
        \|f'(x)-f'(y)\|\le 2 M_\nu + L_\nu \|x-y\|^\nu, \quad \forall x, y \in \dom h
    \end{equation}
    and, for any $\alpha>0$,
    define
    \begin{equation}\label{eq:ML}
        M_f(\alpha):=M_\nu + \frac{L_\nu \alpha}{2}, \quad L_f (\alpha):= L_\nu \nu\left(\frac{1-\nu}{\alpha}\right)^{\frac{1-\nu}{\nu}}.
    \end{equation}
    Then,
        for any $\alpha>0$,
        the pair $(M_f,L_f)=(M_f(\alpha),L_f(\alpha))$ satisfies (A3) and
        \begin{equation}\label{ineq:inf}
            \inf_{\alpha>0} \{ M_f(\alpha)^2 + \bar \varepsilon L_f(\alpha) \} \le 2\left(M_\nu^2+ \bar \varepsilon^{\frac{2\nu}{\nu+1}} L_\nu^{\frac{2}{\nu + 1}}\right).
        \end{equation}
    As a consequence,
\begin{equation}\label{ineq:T}
            T_{\bar\varepsilon}\le \sqrt{2}\left(M_\nu+ \bar \varepsilon^{\frac{\nu}{\nu+1}} L_\nu^{\frac{1}{\nu + 1}}\right).
        \end{equation}
\end{proposition}


We now make two remarks about \eqref{eq:hybrid-holder}.
First, a trivial example of a pair $(f,h)$ satisfying \eqref{eq:hybrid-holder} is $f(\cdot)=M_\nu\|\cdot\|+L_\nu\|\cdot\|^{\nu+1}/(\nu+1)$ and $h \equiv 0$.
More generally, the sum of an $M_f$-Lipschitz continuous function on $\dom h$ and a function whose gradient is $\nu$-H\"{o}lder continuous on $\dom h$ satisfies \eqref{eq:hybrid-holder}.
Second, if \eqref{eq:hybrid-holder} holds with $M_\nu=0$, it follows that
$f$ is differentiable on $\dom h$ and its gradient is $\nu$-H\"{o}lder continuous on $\dom h$.
Algorithms for solving instances of \eqref{eq:ProbIntro} satisfying \eqref{eq:hybrid-holder}
with $M_\nu=0$ have been studied for example in
\cite{lan2015bundle,nesterov2015universal}.

		 Finally,
for a given  tolerance $\bar \varepsilon>0$, it is said that an algorithm for solving \eqref{eq:ProbIntro}
has $\bar \varepsilon$-iteration complexity ${\cal O}(T)$ if its total number of iterations until it obtains a $ \bar \varepsilon $-solution
 is bounded by 
 $C (T+1)$ where $C>0$ is a
universal constant.

	
	\subsection{Review of the CS-CS method}\label{subsec:cs}
	
	We start by reviewing the CS-CS method.
	The CS-CS method with initial point $x _0 \in \dom h$ and constant prox stepsize $\lam>0$, denoted by CS-CS$(x_0,\lam)$,
	recursively computes its iteration sequence $\{x_j\}$
	according to
	\begin{equation}\label{eq:sub}
		x_{j+1} =\underset{u\in  \R^n}\argmin
		\left\lbrace  \ell_f(u;x_{j})+h(u) +\frac{1}{2\lam} \|u- x_{j} \|^2 \right\rbrace
		\qquad \forall j \ge 0.
	\end{equation}
	

	For any given universal constant $C >1$, pair $(M_f,L_f)$ satisfying (A3), and tolerance $\bar \varepsilon>0$, it follows from
	Proposition \ref{prop:sub-new} that 
	CS-CS$(x_0,\lam)$ with any stepsize $\lam$ such that
	$\bar \varepsilon/[4 C (M_f^2+\bar \varepsilon L_f)] \le \lam \le \bar \varepsilon/[4 (M_f^2+\bar \varepsilon L_f)]$,
	has $\bar \varepsilon$-iteration complexity given by
	\begin{equation}\label{eq:bound}
		{\cal O} \left(\min \left\lbrace \frac{ (M_f^2+\bar \varepsilon L_f) d_0^2}{\bar \varepsilon^2}, \left( \frac{M_f^2+\bar \varepsilon L_f}{\mu \bar \varepsilon} + 1 \right) \log\left( \frac{\mu d_0^2}{ \bar \varepsilon} + 1\right) \right\rbrace  +1\right)
	\end{equation}
	(see our slightly modified definition of ${\cal O}(\cdot)$ in Subsection \ref{subsec:DefNot}) with the convention that
	the second term is equal to the first one when $\mu=0$.
	(It is worth noting that the second term 
	converges to the first one as
	$\mu \downarrow 0$.)
	
	In order to obtain the $\bar \varepsilon$-iteration complexity \eqref{eq:bound}, the CS-CS method requires the knowledge of $(M_f, L_f)$ satisfying (A3) to compute a suitable $\lam$.
	Subsection \ref{subsec:A-CS} presents an adaptive variant of the CS-CS method which does not require such knowledge. More precisely, this adaptive variant 
    starts with any stepsize $\lam_0>0$, employs a backtracking procedure to compute a nonincreasing sequence $\{\lam_j\}$ such that each $\lam_j$ satisfies a key condition, and recursively performs iterations similar to \eqref{eq:sub}. It is shown in Proposition \ref{prop:A-CS} that, without the prior knowledge of $T_{\bar \varepsilon}$, the adaptive variant of CS-CS has  $\bar \varepsilon$-iteration complexity given by
	\begin{equation}\label{eq:bound-bar}
    {\cal O}\left(\min \left\lbrace \frac{T_{\bar \varepsilon}^2 d_0^2}{\bar \varepsilon^2}, \left(\frac{T_{\bar \varepsilon}^2}{\mu \bar \varepsilon} + 1\right) \log\left( \frac{\mu d_0^2}{ \bar \varepsilon} + 1\right) \right\rbrace +1 \right).
    \end{equation}
    It is worth noting that bound \eqref{eq:bound-bar} is better than the one for the CS-CS method (i.e., \eqref{eq:bound})
	due to the fact that it is expressed in terms of
	the tighter quantity $T_{\bar \varepsilon}^2$ instead of the estimate $ M_f^2+\bar \varepsilon L_f$.

	\section{The GPB Framework}\label{sec:framework}
	
	This section contains three subsections. Subsection \ref{subsec:update} describes a generic bundle update scheme that is used to perform
	the null iterations of a
	method in the GPB framework.
	Subsection \ref{subsec:GPB} presents the GPB framework and Subsection \ref{subsec:results}
	describes the main complexity results about it.
	
	\subsection{Bundle update schemes}\label{subsec:update}


Bundle methods discussed in the literature rely on different bundle update schemes,
i.e., schemes for updating the bundle function $\Gamma$ in \eqref{eq:x-pre}
	which approximates the objective function of \eqref{eq:ProbIntro}.
	Instead of focusing
	on a specific bundle update scheme, we describe in
	this subsection
	a generic scheme which includes many of the ones
	considered in the literature.
	This subsection also
	gives the details of the three concrete examples (E1)-(E3) of the generic bundle update scheme. 

    We start by describing the bundle update (BU) blackbox. 
    
    \noindent\rule[0.5ex]{1\columnwidth}{1pt}
	
	BU

    \noindent\rule[0.5ex]{1\columnwidth}{1pt}
    {\bf Input:} $(\lam,\tau) \in \R_{++} \times (0,1)$ and
    $(x^c,x,\Gamma) \in \R^n \times
    \R^n \times \mConv{n}$ such that $\Gamma\le \phi$ and
      \eqref{eq:x-pre} holds.
    \begin{itemize}
        \item
        find function $\Gamma^+$ such that
\begin{equation}\label{def:Gamma}
	    \Gamma^+ \in \mConv{n}, \qquad \tau \bar \Gamma(\cdot)  + (1-\tau) [\ell_f(\cdot;x)+h(\cdot)] \le \Gamma^+(\cdot) \le \phi(\cdot),
	\end{equation}
   where $\ell_f(\cdot;\cdot)$ is as in \eqref{def:gamma} and
   $\bar \Gamma(\cdot) $ is such that
\begin{equation}\label{def:bar Gamma} 
	\bar \Gamma \in \mConv{n}, \quad
 \bar \Gamma(x) = \Gamma(x), \quad 
	x = \underset{u\in \R^n}\argmin \left \{\bar \Gamma (u) + \frac{1}{2\lam} \|u-x^c\|^2 \right\}.
	\end{equation}
     \end{itemize}
     {\bf Output:} $\Gamma^+$.
     
    \noindent\rule[0.5ex]{1\columnwidth}{1pt}

	Clearly, the above update scheme does not completely determine $\Gamma^+$
	but rather gives minimal conditions on it which are suitable for the complexity analysis
	of this paper.
	
	We now describe three concrete
 update schemes (E1), (E2), and (E3) which are special ways
 of implementing BU.
 Unless otherwise stated, it
 is assumed that their input is
 the same as in BU.

	\begin{itemize} 
     \item [(E1)] {\bf one-cut scheme}: 
     This scheme obtains $\Gamma^+$ as in \eqref{eq:affine}.
	It is easy to see that if this update is used recursively
	then $\Gamma$ is always of the form
	\begin{equation}\label{eq:Gamma-form}
	    \Gamma(\cdot) = \sum_{x \in X} \alpha_x \ell_f(\cdot;x)  + h(\cdot)
	\end{equation}
	where $X$ is a finite set in $\dom h$ and
  $\{\alpha_x : x \in X\} \subset \R_{++}$ are scalars such that
	$\sum_{x \in X} \alpha_x =1$.
	
        \item [(E2)] \textbf{two-cuts scheme:}
        For this scheme, it is assumed that
        $\Gamma$
        has the form 
        \begin{equation}\label{def:Gamma-E2}
            \Gamma=\max \{A_f,\ell_f(\cdot;x^-)\}+h
        \end{equation}
        where $h \in \mConv{n}$ and
        $A_f$ is an affine function satisfying $A_f\le f$. In view of \eqref{eq:x-pre},
        it can be shown that there exists
        $\theta \in [0,1]$ such that
        \begin{align}
            &\frac1\lam (x-x^c) + \partial h(x) 
+ \theta \nabla A_f + (1-\theta) f'(x^-) \ni 0, \label{theta1} \\
            &\theta A_f(x) + (1-\theta) \ell_f(x;x^-) = \max \{A_f(x),\ell_f(x;x^-)\}. \label{theta2}
        \end{align}
        The scheme then sets
        \begin{equation}\label{def:Af+}
         A_f^+(\cdot) :=  \theta A_f(\cdot) + (1-\theta) \ell_f(\cdot;x^-)
     \end{equation}
     and outputs the function $\Gamma^+$
      defined as
\begin{equation}\label{eq:G-agg}
		    \Gamma^+(\cdot)  := 
		    \max\{A_f^+(\cdot),\ell_f(\cdot;x)\} + h(\cdot).
	    \end{equation}
        


	    \item [(E3)] \textbf{multiple-cuts scheme:} For this scheme, it is assumed that $\Gamma$ of the form
	    $\Gamma=\Gamma(\cdot;B)$ where
	    $B \subset \R^n$ is a finite set (i.e., the current  bundle set) and $\Gamma(\cdot;B)$ is defined as in \eqref{eq:Gamma-E2}.
	     This scheme 
      chooses the next bundle set $B^+$ so that
	    \begin{equation}\label{eq:C+}
        B(x) \cup \{x\} \subset B^+ \subset B \cup \{x\}
        \end{equation}
        where
        \begin{equation}\label{def:C+}
            B(x) := \{ b \in B : \ell_f(x;b)+h(x) = \Gamma(x) \},
        \end{equation}
         and then
	    output $\Gamma^+ = \Gamma(\cdot;B^+)$.

	\end{itemize}

 It is interesting to note that
 \eqref{eq:G-agg}, \eqref{eq:C+} and the
 definition of $\Gamma^+$ in (E3) imply that
        the  updates $\Gamma^+$ output by
        schemes (E2) and (E3) have the property that $\Gamma^+(\cdot)$ is minorized 
        by $\ell_f(\cdot;x)+h(\cdot)$
        where $x$ is as in \eqref{eq:x-pre}.
        On the other hand,
        $\Gamma^+$ output by (E1)
        does not necessarily has this property.

We now make some remarks to argue that all the update schemes 
above are special implementations of BU.
It can be easily seen that 
the update $\Gamma^+$ in
(E1), together with
 $\bar \Gamma=\Gamma$, satisfies
 \eqref{def:Gamma} and \eqref{def:bar Gamma}, and hence
 that this $\Gamma^+$ is a special way of implementing BU.
 On the other hand, the proofs that
 the updates $\Gamma^+$ of (E2) and (E3) are
 special implementations of BU
 are more involved and are given in
 Propositions \ref{lem:E2} and \ref{lem:E3}, respectively.

	\subsection{The GPB framework}\label{subsec:GPB}
	
This subsection states the GPB framework based on the BU blackbox presented in Subsection \ref{subsec:update}. It also
gives several remarks about GPB and discusses how it relates
to the classical
proximal point method.

Before stating GPB, we first give
a brief description for its
$j$-th iteration.
Given a prox-center $x^c_{j}$,
it attempts to approximately solve
the prox subproblem
\begin{equation}\label{eq:key-prox}
    m_{j}^*:=\min \left\{ \phi(u)+\frac1{2\lam}\|u-x_{j}^c\|^2 : u\in \R^n\right\}
\end{equation}
(according to a certain termination
criterion outlined below)
by computing the exact solution
$x_j$ of 
the  approximate prox subproblem of the
form \eqref{eq:x-pre} with $x^c=x^c_j$ and
with bundle function $\Gamma=\Gamma_j$
obtained for example according to
one of the update schemes (E1), (E2) or (E3) described above.
If it succeeds then 
$x_{j+1}^c$ is set
to be $x_{j}$;
otherwise,
$x_{j+1}^c$ is set to be $x_{j}^c$.
Finally,
$j$ is updated to
$j+1$ and the above iteration is
repeated.

The method outlined above can be viewed as an inexact proximal point method. More specifically,
consecutive iterations
$j$ such that
$x_{j}^c$ remains the same 
approximately solve the prox
subproblem \eqref{eq:key-prox} (which does not depend on $j$).
When that happens at an iteration
$j$, the prox-center for the next iteration $j+1$ is then updated to a new one.

We now describe the aforementioned
termination criterion.
Given $\delta>0$,
it checks whether
$x_{j}$ and the iterate $y_{j}$ defined as 
\begin{equation}\label{eq:yj-intuition}
    y_j \in \Argmin \left\lbrace \phi(x) :
		x \in \{x_0,x_1,\ldots,x_j\}
		\right\rbrace. 
\end{equation}
satisfies
\[
t_j:= \phi(y_j) - \Gamma_j(x_j) - \frac{1}{2\lam} \|x_j-x^c_j\|^2 \le \delta.
\]




We are now ready to state GPB.

	
	
	\noindent\rule[0.5ex]{1\columnwidth}{1pt}
	
	GPB
	
	\noindent\rule[0.5ex]{1\columnwidth}{1pt}
	\begin{itemize}
		\item [0.] Let $ x_0\in \dom h $, $\lambda>0$, $ \bar \varepsilon>0 $ and $\tau \in (0,1)$ be given
		such that
        \begin{equation}\label{rel:tau1} 
	        \frac{\tau}{1-\tau} \ge \frac{8\lam T_{\bar \varepsilon}^2}{(1+\lam\mu)\bar \varepsilon}
	    \end{equation}
		where $T_{\bar \varepsilon}$ is as in \eqref{def:T}, and set  
		$y_0=x_0$, $t_0=0$ and $j=1$;
		\item[1.] if $t_{j-1}\le \bar \varepsilon/2$,
		 then
			perform a {\bf serious update}, i.e., set $x^c_{j} = x_{j-1}$ and find $\Gamma_{j}$ such that
		\begin{equation}\label{ineq:require}
		    \Gamma_{j}\in \mConv{n}, \quad \ell_f(\cdot;x_{j-1})+h \le \Gamma_{j}\le  \phi;
		\end{equation}
		else, perform a {\bf null update}, i.e., set $x^c_{j}= x^c_{j-1}$
  and let $\Gamma_{j}$ be the output of the BU blackbox with input $(\lam,\tau)$ and $(x^c,x,\Gamma) = (x_{j-1}^c, x_{j-1},\Gamma_{j-1})$;
		
	    \item[2.]
	    compute 
		\begin{equation}
	    x_{j} =\underset{u\in \R^n}\argmin
	    \left\lbrace \Gamma_{j}(u) +\frac{1}{2\lam}\|u- x_{j}^c \|^2 \right\rbrace,  \label{def:xj}
	    \end{equation}
		 choose $y_j$ according
   to \eqref{eq:yj-intuition},
		and set  
		\begin{equation}\label{ineq:hpe1}
		    m_{j}=\Gamma_{j}(x_{j}) +\frac{1}{2\lam}\|x_{j}- x_{j}^c \|^2, \qquad t_{j} = \phi(y_{j}) - m_{j};
		\end{equation}
		\item[3.] set $j\leftarrow j+1$ and go to step 1.
	\end{itemize}
	\rule[0.5ex]{1\columnwidth}{1pt}









	An iteration $j$ such that $t_j \le \bar \varepsilon/2$
     is called
	a serious iteration in which
	case $x_j$ (resp., $y_j$) is called a serious iterate
	(resp., auxiliary serious iterate);
	otherwise, $j$ is called a null iteration.
 Let $
 j_1 \le j_2 \le \ldots$ denote the sequence of all serious iterations and
 define the $k$-th cycle ${\cal C}_k$ to be the iterations $j$ such that $j_{k-1}+1 \le j \le j_k$, i.e.,
 \begin{equation}\label{def:Ck}
     {\cal C}_k := \{ j_{k-1}+1,\ldots,j_k\}
 \end{equation}
 where $j_0:=0$.
 Hence, only the last 
 iteration of a cycle (which can be the first one
 if ${\cal C}_k$ contains only one iteration)  
 is serious. 

	We make some basic remarks about GPB.
	First, we refer to it as a framework since it does not completely
	specify how some algorithmic quantities are generated.
	The framework rather gives minimal  conditions on these quantities
	which enables us to establish complexity bounds for all specific instances of it
 in a unified manner.
    Second, in view of \eqref{ineq:require} or the fact that the output of BU satisfies \eqref{def:Gamma}, it follows that
    \begin{equation}\label{phi-property}
        \Gamma_j \le \phi, \quad
    \Gamma_j \in \mConv{n} \quad \forall j \ge 1.
    \end{equation}
    Third,
    in view of the definition of ${\cal C}_k$ and the way
 the prox-center iterates are generated, it is easy to see that for every $k \ge 1$, we have
 \begin{equation}\label{eq:connect}
     x^c_j = x_{j_{k-1}}
 \quad
 \forall j \in {\cal C}_k.
 \end{equation}
 In words, all prox-centers in
 the $k$-th cycle is equal to the most recent serious iterate.
	Fourth, schemes (E1)-(E3)
 in the previous subsection provide three possible concrete ways of implementing
 the BU blackbox in step 1. 
%
	Fifth, although GPB does not specify a termination criterion for the sake
	of shortness, all iteration-complexity bounds established in this paper are
	relative to the effort of obtaining a $\bar \varepsilon$-solution of \eqref{eq:ProbIntro}.
	Finally, although iteration-complexity bounds for GPB can also be established for other termination criteria (see for example
	Section 6 of \cite{liang2020proximal}),
	we have omitted the details of
	their derivation for the sake of
	shortness.

 We now make some observations about possible simple ways of choosing the bundle function $\Gamma_{j}$ in a serious update. Specifically,
 two simple ways are: 1) $\Gamma_{j}=\ell_f(\cdot;x_{j-1})+h$, and 2) $\Gamma_{j}=\max\{\Gamma_{j-1}, \ell_f(\cdot;x_{j-1})+h\}$.
 Moreover, under the assumption that every call to BU during a null update is carried out using
 (E2) (resp., (E3)), another
 way to obtain $\Gamma_{j}$ during a serious update is  to also use update (E2) (resp., (E3)). In view of the observation in the second last paragraph in Subsection \ref{subsec:update}, it follows that the latter
 way yields a bundle function $\Gamma_{j}$
 satisfying \eqref{ineq:require}.

	We now discuss the role played by the parameter $\tau$ of GPB.
	First, $\tau$ is only used in step 1 as input to the BU blackbox to obtain $\Gamma_{j}$.
 Second,
even though the analysis of GPB
depends on a scalar $\tau$
satisfying \eqref{rel:tau1}, the implementations of some specific instances of GPB do not require knowledge of
such $\tau$. For instance,
since the updates (E2) and (E3) do not depend on $\tau$,
the GPB instances
2C-PB and MC-PB do not depend
on $\tau$ either.
	(Recall the meaning of 1C-PB, 2C-PB and MC-PB given in the sentence following (E3) in Section \ref{sec:intro}.)
	Third, the GPB instance
	1C-PB requires a scalar $\tau$ satisfying
 \eqref{rel:tau1} since the
	update (E1) depends on $\tau$ (see \eqref{eq:affine}). 
	Finally, \eqref{rel:tau1} implies that $\tau$ has to be sufficiently close to one
 which, in the context of
 (E1), means that the new bundle $\Gamma_j$ is closer to $\Gamma_{j-1}$
 than the new cut
 $\ell_f(\cdot;x_{j-1}) + h(\cdot)$
 in view of the nature of the one-cut scheme (E1) (see relation \eqref{eq:affine}).

 

	We finally briefly discuss how accurately
 GPB
 solves the prox problem \eqref{eq:key-prox}.
 Since $\Gamma_{j} \le \phi$, it follows from the definition of $m_{j}$ in \eqref{ineq:hpe1} that $m_{j} \le m_{j}^*$, and hence that
 \begin{align}\label{ineq:tilde-tj}
     0 &\le \phi(y_{j})+\frac1{2\lam}\|y_{j}-x_{j}^c\|^2 - m_{j}^* \nn \\
     &\le \phi(y_{j})+\frac1{2\lam}\|y_{j}-x_{j}^c\|^2 - m_{j}
  = t_{j} + \frac{1}{2\lam} \|y_{j}-x_{j}^c\|^2.
 \end{align}
Thus, if ${j}$ is a serious iteration, or equivalently,
	$t_{j}\le \bar \varepsilon/2$, it follows that
	$y_{j}$ is a $\bar \varepsilon_{j}$-solution of \eqref{eq:key-prox}
 where
 \[
 \bar \varepsilon_{j} := \frac{\bar \varepsilon}2 + \frac{1}{2\lam} \|y_{j}-x_{j}^c\|^2.
 \]
	The sequence of consecutive null iterations between two serious ones can be regarded as an iterative procedure to compute the aforementioned $\bar \varepsilon_{j}$-solution. 
 More details of such an interpretation can be found in Subsection 3.1 of \cite{liang2020proximal}.
 
 Observe that even though the right-hand side of \eqref{ineq:tilde-tj} contains two terms,
 our serious step condition used in GPB only checks the magnitude of the first one. It is possible to modify GPB to one whose serious step condition controls
 the magnitude of
 the right-hand side of \eqref{ineq:tilde-tj}. However, since the latter serious step condition is more restrictive, the resulting method will perform more null iterations, and hence its practical performance might not be as good as the one proposed in this paper.

	
We end this subsection
by stating a general complexity
bound which applies to
any GPB variant.
It assumes that
the triple $(M_f, L_f, \mu)$
is known so that
a parameter $\tau$
satisfying \eqref{rel:tau1} can be computed.

	\begin{theorem}\label{thm:bound-cvx}
	    Let universal constant $C >0$, initial point $x_0 \in \dom h$, tolerance $\bar \varepsilon>0$, and
	    instance $(f,f';h)$ of \eqref{eq:ProbIntro}
		satisfying (A1)-(A3) for some parameter triple $(L_f, M_f,\mu) \in \R^3_+$
		be given.
		Then, if $\lam$ satisfies
		\begin{equation}\label{ineq:lam1}
		\frac{\bar \varepsilon}{C (M_f^2+\bar \varepsilon L_f)} \le \lam \le \frac{C d_0^2}{\bar \varepsilon},
		\end{equation}
		and $\tau$ is given by
		\begin{equation}\label{eq:tau}
		    \tau = \left[ 1+\frac{(1+\lam \mu)\bar \varepsilon}{8\lam(M_f^2 + \bar \varepsilon L_f)}\right]^{-1},
		\end{equation}
		then
		any variant of
		GPB with input ($x_0,\lam, \bar \varepsilon,\tau$)
		obtains a $\bar \varepsilon$-solution of the above instance
		in a number of iterations bounded (up to a logarithmic term) by \eqref{eq:bound-mu}.
	\end{theorem}

	\subsection{Iteration-complexity results for
 $\tau$-free GPB variants}\label{subsec:results}

This subsection considers the
subclass of GPB  methods,
referred to as the
$\tau$-free GPB subclass,
which do not
depend on $\tau$ (and hence do not need $\tau$
as input),
and derives improved iteration-complexity bounds for it
which follow as immediate consequences of Theorem \ref{thm:bound-cvx}.
Since 2C-PB and MC-PB do not depend on $\tau$,
the results below apply to both of them.



\begin{corollary}\label{cor:E2E3}
	    Let universal constant $C >0$, initial point $x_0 \in \dom h$, and tolerance $\bar \varepsilon>0$ be given, and consider an instance $(f,f';h)$ of \eqref{eq:ProbIntro} satisfying (A1)-(A3).
		Then, any variant of the
		$\tau$-free GPB subclass with
		input $(x_0,\lam,\bar \varepsilon)$ 
		satisfying
		\begin{equation}\label{ineq:lam2}
		    \frac{\bar \varepsilon}{C T_{\bar \varepsilon}^2} \le \lam \le \frac{C d_0^2}{\bar \varepsilon}
		\end{equation}
		where $T_{\bar \varepsilon}$ is as in \eqref{def:T}
		obtains a $\bar \varepsilon$-solution of the above instance
		in a number of iterations bounded (up to a logarithmic term) by \eqref{eq:bound-bar}.
	\end{corollary}

	\begin{proof}
	    Observe that any variant of the $\tau$-free GPB subclass can be viewed as
	    an instance of GPB with
	    input $\tau$ satisfying the equality in \eqref{rel:tau1} since it does not depend on $\tau$.
	    Hence, it follows from \eqref{eq:bound-bar} and
	    Theorem \ref{thm:bound-cvx} 
	   with $(L_f,M_f)$ replaced by $(\bar L_f, \bar M_f)$
	   that the conclusion of the corollary holds.
	\end{proof}
	

Recall that Proposition \ref{lem:hybrid} shows that if $f$ satisfies \eqref{eq:hybrid-holder} then it satisfies (A3) with $(M_f,L_f)=(M_f(\alpha),L_f(\alpha))$. The following result is a consequence of Corollary \ref{cor:E2E3} when condition \eqref{eq:hybrid-holder} holds in place of (A3). We omit its proof since it directly follows from Corollary \ref{cor:E2E3} and \eqref{ineq:T}.

\begin{corollary}\label{cor:holder}
	    Let universal constant $C >0$, initial point $x_0 \in \dom h$, and tolerance $\bar \varepsilon>0$ be given, and consider an instance $(f,f';h)$ of \eqref{eq:ProbIntro} such that (A1), (A2), and condition \eqref{eq:hybrid-holder} hold for some
        quadruple $(M_\nu,L_\nu, \mu, \nu)\in \R^3_+\times (0,1)$.
		Then, any variant of the
		$\tau$-free GPB subclass with
		input $(x_0,\lam,\bar \varepsilon)$ 
		satisfying \eqref{ineq:lam2}
		obtains a $\bar \varepsilon$-solution of the above instance
		in a number of iterations bounded (up to a logarithmic term) by 
        \begin{equation}\label{cmplx:holder}
            {\cal O}\left(\min \left\lbrace \frac{M_\nu^2 d_0^2}{\bar \varepsilon^2} + \left(\frac{L_\nu}{\bar \varepsilon}\right)^{\frac{2}{\nu + 1}} d_0^2, \left(\frac{M_\nu^2}{\mu \bar \varepsilon} + \frac{L_\nu^{\frac{2}{\nu+1}}}{\mu \bar \varepsilon^{\frac{1-\nu}{1+\nu}}} + 1\right) \log\left( \frac{\mu d_0^2}{ \bar \varepsilon} + 1\right) \right\rbrace +1 \right).
        \end{equation}
	\end{corollary}

We now make two remarks about Corollary \ref{cor:holder}.
First, when $M_\nu=0$,
bound \eqref{cmplx:holder}
reduces to
\[
        {\cal O}\left(\min \left\lbrace \left(\frac{L_\nu}{\bar \varepsilon}\right)^{\frac{2}{\nu + 1}} d_0^2, \left( \frac{L_\nu^{\frac{2}{\nu+1}}}{\mu \bar \varepsilon^{\frac{1-\nu}{1+\nu}}} + 1\right) \log\left( \frac{\mu d_0^2}{ \bar \varepsilon} + 1\right) \right\rbrace  +1\right).
        \]
Second, when $\mu=0$,
the above bound
 agrees with
the one obtained  for the primal universal method of \cite{nesterov2015universal} (see (2.20) therein).

   For the sake of comparing the results of this paper with
   the ones obtained in \cite{liang2020proximal}, we now state
   another consequence of Theorem \ref{thm:bound-cvx} in which
   an alternative $\bar \varepsilon$-iteration complexity 
   for $\tau$-free GPB instances applied to instances
   of \eqref{eq:ProbIntro} with $\bar M_{f,0}$ finite.
   (Recall the definition of $\bar M_{f,0}$ is in the line below \eqref{def:T}.)
    
    \begin{corollary}\label{cor:nonsmooth}
    Let universal constant $C >0$, initial point $x_0\in \dom h$, and tolerance $\bar \varepsilon>0$ be given,
    and consider an instance $(f,f';h)$ of \eqref{eq:ProbIntro} such that (A1)-(A3) holds and
    $\bar M_{f,0}$ is finite.
			Then, any variant of the
		$\tau$-free GPB subclass with
		input $(x_0,\lam,\bar \varepsilon)$ 
		satisfying
		\begin{equation}\label{ineq:lam2-M}
		    \frac{\bar \varepsilon}{C (\bar M_{f,0})^2} \le \lam \le \frac{C d_0^2}{\bar \varepsilon},
		\end{equation}
		obtains a $\bar \varepsilon$-solution of the above instance
		in a number of iterations bounded (up to a logarithmic term) by 
		\begin{equation}\label{eq:bound-bar M}
    {\cal O} \left(\min \left\lbrace \frac{(\bar M_{f,0})^2 d_0^2}{\bar \varepsilon^2}, \left(\frac{(\bar M_{f,0})^2}{\mu \bar \varepsilon} + 1\right) \log\left( \frac{\mu d_0^2}{ \bar \varepsilon} + 1\right) \right\rbrace +1 \right).
    \end{equation}
	\end{corollary}
    
    \begin{proof}
	    Observe that any variant of the $\tau$-free GPB subclass can be viewed as
	    an instance of GPB with
	    input $\tau=[1+(1+\lam \mu)\bar \varepsilon/(8\lam \bar M_{f,0}^2)]^{-1}$ since it does not depend on $\tau$.
	    Since the pair $(0,\bar M_{f,0})$ satisfies conditions
	    (A1)-(A3), it then follows from \eqref{eq:bound-bar} with $T_{\bar \varepsilon}$ replaced by $\bar M_{f,0}$ and
	    Theorem \ref{thm:bound-cvx} 
	   with $(L_f,M_f)$ replaced by $(0,\bar M_{f,0})$
	   that the conclusion of the corollary holds.
	\end{proof}

    Before comparing the RPB method of \cite{liang2020proximal}
    with the $\tau$-free GPB instances of the paper, we first make two remarks about the first one in regards to the latter ones.
    First, 
    the RPB method of \cite{liang2020proximal} with $\delta=\bar \varepsilon/2$ can be viewed as a special case of $\tau$-free GPB since:
    RPB uses the inequality $t_{j-1}\le \bar \varepsilon/2$ to decide
    whether to
	perform a serious or null update; and,
	its serious and null updates, the
	latter of which are based on (E3), fulfill the requirements of step 1 of GPB (see Lemma \ref{lem:E2}).
	Second, while the RPB method of \cite{liang2020proximal} deals with
	instances of \eqref{eq:ProbIntro} such that $\bar M_{f,0}$ is finite (i.e., the nonsmooth setting),
	the analysis presented in
	this paper for $\tau$-free GPB applies to the
	larger class of
	instances of \eqref{eq:ProbIntro} such that
	$T_{\bar \varepsilon}$ is finite
	(i.e., the hybrid or smooth/nonsmooth setting).

    We now compare Corollary \ref{cor:nonsmooth} of this paper
    with Corollary 3.2 of \cite{liang2020proximal}.
    Indeed, it follows from Corollary 3.2 of \cite{liang2020proximal} with $M_f=\bar M_{f,0}$
    that RPB has $\bar \varepsilon$-iteration complexity given by \eqref{eq:bound-bar M} as long 
    $d_0/\bar M_{f,0} \le \lam \le C d_0^2/ \bar \varepsilon$
    and
    $\mu \le C \bar M_{f,0}/d_0$.
    On the other hand, Corollary \ref{cor:nonsmooth} of this paper establishes complexity bound \eqref{eq:bound-bar M} for any $\lam$ lying in the larger range \eqref{ineq:lam2-M} without imposing any condition on $\mu$.

We now compare Corollary \ref{cor:nonsmooth} of this paper
    with Corollary 3.3 of \cite{liang2020proximal}.
Indeed, it follows from Corollary 3.3 of \cite{liang2020proximal} with $M_f=\bar M_{f,0}$ that
    RPB has $\bar \varepsilon$-iteration complexity $ {\cal O}( (\bar M_{f,0})^2 d_0^2/ \bar \varepsilon^2 +1 )$ as long \eqref{ineq:lam2-M} holds 
    and
    $h$ is $(C \bar M_{f,0})$-Lipschitz continuous.
    On the other hand, Corollary \ref{cor:nonsmooth} of this paper establishes the (possibly sharper)
    $\bar \varepsilon$-iteration complexity \eqref{eq:bound-bar M} for any $\lam$ in the
    same range without imposing any condition Lipschitz continuity on $h$.

     Even though 1C-PB depends on $\tau$,
    it can be easily seen that its iteration-complexity is similar to the one
    of Corollary \ref{cor:E2E3} if
    $\tau$ is close to the one satisfying the equality in \eqref{rel:tau1}.
    Section \ref{sec:adap} describes an adaptive variant of 1C-PB 
    which adaptively chooses
    $\tau=\tau_j$
    such that
    a key condition holds in every iteration $j$
    and which 
    has the same $\bar \varepsilon$-iteration complexity as that of Corollary \ref{cor:E2E3}. 


    

	\section{Complexity Analysis of GPB}\label{sec:analysis}

	This section consists of three subsections.
	The first one provides a bound on the number of serious iterates generated by the GPB framework.
	The second one derives a preliminary complexity bound on the number of possible consecutive null iterates.
	Finally, the last subsection combines the aforementioned bounds to obtain a complexity bound on the total number of iterations performed by any algorithm in the GPB framework with prox stepsize
	$\lam$ arbitrarily chosen.
	Moreover, it also provides the proof of	Theorem \ref{thm:bound-cvx} as a consequence of this general complexity result.

	\subsection{Bounding the number of serious iterates}\label{subsec:serious}
	
	We start by introducing
	some notation and definitions.
	Recall from the paragraph following GPB that $j_1 < j_2 < \ldots$ denote the serious iterations of the GPB framework. Now,
	define $ \hat x_0:=x_0 $, and for every $k \ge 1$, let 
	\begin{align}
		\hat x_k &:= x_{j_k}, \quad \hat y_k := y_{j_k}, \quad \hat \Gamma_k:=\Gamma_{j_k},
		\quad \hat m_k := m_{j_k}. \label{not}
	\end{align}
	
	The following result summarizes the basic properties of the above ``hat'' entities that
	follow as an immediate consequence of their definitions and the description of the
	GPB framework. It is worth noting that the complexity results developed in this subsection
	apply not only to the sequences defined in \eqref{not}, but also to arbitrary sequences
	$\{\hat x_k\}$, $\{\hat y_k\}$ and $\{\hat \Gamma_k\}$ satisfying the basic properties stated below.

    \begin{lemma}\label{lem:iterate}
		The following statements about GPB hold for every $ k\ge 1 $:
		\begin{itemize}
		    \item[a)] $\hat \Gamma_k \in \mConv{n}$ and $\hat \Gamma_k \le \phi$;
			\item[b)]  $(\hat x_k,\hat m_k)$ is the pair of optimal solution and optimal value of
			\[
			\min \left\lbrace \hat \Gamma_k(u) + \frac{1}{2\lam}\|u- \hat x_{k-1} \|^2: u\in\R^n \right\rbrace;
			\]
			\item[c)] 	there holds $\phi(\hat y_k) - \hat m_k \le \bar \varepsilon/2$.
		\end{itemize}
	\end{lemma}
	
    \begin{proof}
        a) This statement follows from 
\eqref{phi-property} and the definition of $\hat \Gamma_k$ in \eqref{not}.
		
		b) It follows from \eqref{def:xj} with $j=j_k$, the first identity in \eqref{ineq:hpe1} with $j=j_k$, and relations \eqref{eq:connect} and \eqref{not}, that b) holds.
		
		c) Since $ j_k $ is a serious iteration, we have that $ t_{j_k}\le \bar \varepsilon/2 $. Using this conclusion, \eqref{not}, and the definition of $t_j$ in \eqref{ineq:hpe1}, we conclude that c) holds.
	\end{proof}
	
	It is worth noting that a), b), and c) can be viewed only as properties about
	the sequences $\{\hat \Gamma_k\}$ and $\{\hat y_k\}$, and the initial point $\hat x_0$, since
	$\{\hat x_k: k \ge 1\}$ is uniquely determined by $\{\hat \Gamma_k\}$.
	
	The next result provides an important recursive formula for the sequences in \eqref{not}
	and derives some important consequences that follow from it.
	
	\begin{lemma}\label{lem:iterate2}
	Let $u \in \dom h$ be given and define
	\begin{equation}\label{eq:lam-mu}
	   \lam_\mu = \frac{\lam}{1+\lam \mu}.
	\end{equation}
	Then, the following statements hold:
	\begin{itemize}
	    \item [a)]
	    for every $k \ge 1$, we have
	    \begin{equation}\label{ineq:sat}
	        \phi(\hat y_k) - \phi(u) \le \frac{1}{2\lam}\|\hat x_{k-1}-u\|^2 - \frac{1}{2\lam_\mu}\|\hat x_{k} - u\|^2 + \frac{\bar \varepsilon}{2};
	    \end{equation}
	    \item[b)] we have $\min_{1\le k \le K} \{ \phi(\hat y_k)-\phi(u)\} \le \bar \varepsilon$ for every index $K$  satisfying
	    \[
		K \ge \min \left\lbrace \frac{\|x_0-u\|^2}{\lam \bar \varepsilon}, \frac{1}{\mu \lam_\mu} \log\left(  \frac{\mu \|x_0-u\|^2}{\bar \varepsilon}+1 \right) \right\rbrace;
		\]
		\item[c)] for every $k \ge 1$, we have $\|\hat x_k-u\|^2 \le \|x_0-u\|^2 +  \lam k \bar \varepsilon$.
	\end{itemize}
	\end{lemma}

	
	\begin{proof}
	    a) It follows from Lemma \ref{lem:iterate}(a) that $\hat \Gamma_k$ is $\mu$-convex, and hence that the objective function in Lemma \ref{lem:iterate}(b) is $ (\mu + 1/\lam)$-strongly convex.
		Using this observation, Lemma \ref{lem:iterate}(b) and Theorem 5.25(b) of \cite{beck2017first}
		with $f=\hat \Gamma_k+\|\cdot-\hat x_{k-1}\|^2/(2\lam)$, $x^*=\hat x_k$ and $ \sigma=\mu + 1/\lam $,
		we have for the given $ u\in \dom h $ and every $k \ge 1$,
		\begin{equation}\label{ineq:inter}
		    \hat m_k + \frac1{2}\left( \mu+\frac{1}{\lam}\right) \|u-\hat x_{k}\|^2
		\le \hat\Gamma_k(u) + \frac1{2\lam}\|u-\hat x_{k-1}\|^2.
		\end{equation}
		Using the above inequality and Lemma \ref{lem:iterate}(a) and (c), we
		conclude that 
 \begin{align*}
		    \phi(\hat y_k) &- \phi(u) + \frac1{2}\left( \mu+\frac{1}{\lam}\right)\|\hat x_{k} - u\|^2 
			\le \phi(\hat y_k) - \hat\Gamma_k(u) + \frac1{2}\left( \mu+\frac{1}{\lam}\right)\|\hat x_{k} - u\|^2\\
			&\overset{\eqref{ineq:inter}}{\le} \phi(\hat y_k) - \hat m_k + \frac1{2\lam}\|u-\hat x_{k-1}\|^2
			\le \frac{\bar \varepsilon}2  + \frac{1}{2\lam}\|u-\hat x_{k-1}\|^2
		\end{align*}
		and hence that a) holds.
		
		b)-c) Since \eqref{ineq:sat} is a special case of  inequality \eqref{eq:easyrecur} in which
		\[
		\eta_k=\phi(\hat y_k) - \phi(u), \quad \alpha_k=\frac{1}{2\lam}\|\hat x_{k}-u\|^2, \quad \theta=1+\lam\mu, \quad \delta= \frac{\bar \varepsilon}2,
		\]
	it follows from Corollary \ref{lm:easyrecur}, the fact that $\hat x_0=x_0$ and the definition of $\lam_\mu$ in \eqref{eq:lam-mu} that b) and c) hold.
	\end{proof}

    We are now ready to present the main result of this subsection which provides a
    bound on the number of serious iterates generated by GPB until it obtains
    a $\bar \varepsilon$-solution of \eqref{eq:ProbIntro}.
    	
	\begin{proposition}\label{prop:outer}
	    The number of serious iterations $K$ performed by GPB until it obtains for the
	    first time an auxiliary serious iterate $\hat y_K$ such that $ \phi(\hat y_K)-\phi^*\le \bar \varepsilon $ is bounded by
				\begin{equation}\label{bound:outer}
				    \min \left\lbrace \frac{d_0^2}{\lam \bar \varepsilon}, \frac{1}{\mu \lam_\mu} \log\left(  \frac{\mu d_0^2}{\bar \varepsilon}+1 \right) \right\rbrace + 1
				\end{equation}
			where $\lam_\mu$ is as in \eqref{eq:lam-mu}. Moreover, 
			\begin{equation}\label{ineq:dist}
			    \|\hat x_k-x_0^*\|\le \sqrt{2} d_0 \qquad \forall k \in \{0, 1,\ldots, K-1\}.
			\end{equation}
	\end{proposition}
	
	\begin{proof}
    Lemma \ref{lem:iterate2}(b) with $u=x_0^*$ and the definition of $d_0$ in \eqref{def:d0} imply the first conclusion of the proposition, and hence that $K-1\le d_0^2/(\lam \bar \varepsilon)$.
    This conclusion, together with Lemma \ref{lem:iterate2}(c) with $u=x_0^*$, then implies \eqref{ineq:dist}.
	\end{proof}
	
	We note that Proposition \ref{prop:outer} holds for any $\lam>0$.
	
	\subsection{Bounding the number of consecutive null iterates}\label{subsec:null}
	
	Our goal in this subsection is to show that the set ${\cal C}_k$ is finite and also to provide a bound on its cardinality in terms of $\bar M_f$, $\bar L_f$, $\lam$, $\bar \varepsilon$, $d_0$, and $\tau $.
	
	We start by noting that \eqref{eq:connect}, the definition of $m_j$  in \eqref{ineq:hpe1}, and  the first identity in \eqref{not}, imply that
	\begin{equation}\label{basic obs ii}
		m_j = 
		\Gamma_{j}(x_j) +\frac{1}{2\lam}\|x_j- \hat x_{k-1} \|^2 \quad \forall j \in {\cal C}_k.
	\end{equation}

	
	The first result below describes some
 basic properties of a
 sequence of auxiliary bundle functions $\{\bar \Gamma_j\}$ whose existence is guaranteed by the nature of the BU blackbox.
	
	\begin{lemma}\label{lem:101}
	    For every $j \in {\cal C}_k\setminus \{j_k\}$, the following statements hold:
	    \begin{itemize}
	        \item[a)] there exists function $\bar \Gamma_{j}(\cdot)$ such that
	        \begin{align}
	            &\tau \bar \Gamma_{j}(\cdot) + (1-\tau) [\ell_f(\cdot;x_j)+h(\cdot)] \le \Gamma_{j+1}(\cdot) \le \phi(\cdot), \label{eq:Gamma_j} \\
	            &\bar \Gamma_{j} \in \mConv{n}, \quad \bar \Gamma_{j}(x_j) = \Gamma_{j}(x_j), \quad 
	        x_j = \underset{u\in \R^n}\argmin \left \{\bar \Gamma_{j} (u) + \frac{1}{2\lam} \|u-\hat x_{k-1}\|^2 \right\}; \label{eq:relation}
	        \end{align} 
	        \item[b)] if $\lam_\mu$ is as in \eqref{eq:lam-mu}, then for every $u\in \R^n$, we have
       \begin{equation}\label{ineq:Gammaj}
           \bar \Gamma_{j}(u) + \frac1{2\lam}\|u-\hat x_{k-1}\|^2\ge m_j+ \frac1{2 \lam_\mu}\|u-x_j\|^2 .
       \end{equation}
	    \end{itemize}
	\end{lemma}
	\begin{proof}
	    a) This statement immediately follows from \eqref{def:Gamma}, \eqref{def:bar Gamma}, and the facts that $\Gamma_{j+1}$ is the output of the BU blackbox with input $(\lam,\tau)$ and $(x^c,x,\Gamma) = (x_{j}^c, x_j,\Gamma_j)$ (see the null update in step 1 of GPB) and $x_{j}^c=\hat x_{k-1}$. 
	    
	    b) It follows from $\bar \Gamma_{j}\in \mConv{n}$
     and the definition of $\lam_\mu$ in \eqref{eq:lam-mu} that $\bar \Gamma_{j} + \|\cdot-\hat x_{k-1}\|^2/(2\lam) $ is $(\lam_\mu^{-1})$-strongly convex. 
	    Using the second identity in \eqref{eq:relation} and Theorem 5.25(b) of \cite{beck2017first} with $ f=\bar \Gamma_{j} + \|\cdot-\hat x_{k-1}\|^2/(2\lam) $, $ x^*=x_j $ and $ \sigma= \lam_\mu^{-1} $, we have for every $u\in\dom h$,
	    \[
	    \bar \Gamma_{j}(u) + \frac{1}{2\lam} \|u-\hat x_{k-1}\|^2  \ge \bar \Gamma_{j}(x_j) + \frac{1}{2\lam} \|x_j-\hat x_{k-1}\|^2  + \frac{1}{2 \lam_\mu}\|u-x_j\|^2.
	    \]
	    The statement follows from the above inequality, the first identity in \eqref{eq:relation}, and relation \eqref{basic obs ii}.
	\end{proof}
	
	The following technical result provides an important recursive formula for $\{m_j\}$
	which is used in Lemma \ref{lem:tj} to give
	a recursive formula for $\{t_j\}$.
	It is worth observing that its proof
	uses for the first time the condition
	 \eqref{rel:tau1}.

	\begin{lemma}\label{lem:recur}
		Suppose \eqref{rel:tau1} holds, then for every $j\in {\cal C}_k\setminus \{j_k\}$, 
        we have 
		\begin{equation}\label{ineq:mj}
		    m_{j+1} \ge \tau m_j + (1-\tau) \left[ \ell_f(x_{j+1};x_j) + h(x_{j+1}) + \left(\frac{\bar L_f}2+\frac{4\bar M_f^2}{\bar \varepsilon}\right) \|x_{j+1}-x_j\|^2 \right].
		\end{equation}
	\end{lemma}
    
	\begin{proof}
First, it immediately follows from \eqref{rel:tau1} and the definitions of $T_{\bar \varepsilon}$ and $\lam_\mu$ in \eqref{def:T} and \eqref{eq:lam-mu}, respectively, that 
	    \begin{equation}\label{rel:tau} 
	        \frac{\tau}{1-\tau} \ge \frac{8\lam(\bar M_f^2 + \bar \varepsilon \bar L_f)}{(1+\lam\mu)\bar \varepsilon} \ge \lam_{\mu} \left(\bar L_f + \frac{8\bar M_f^2}{\bar \varepsilon}\right).
	    \end{equation}
		Using \eqref{basic obs ii}, \eqref{eq:Gamma_j}, the fact that $\tau<1$, and \eqref{ineq:Gammaj} with $u=x_{j+1}$, we have
		\begin{align*}
		m_{j+1} & \overset{\eqref{basic obs ii}}{=} \Gamma_{j+1}(x_{j+1}) + \frac{1}{2\lam} \|x_{j+1}-\hat x_{k-1}\|^2\\
		&\overset{\eqref{eq:Gamma_j}}{\ge} (1-\tau) [\ell_f(x_{j+1};x_j) + h(x_{j+1})] + \tau \left( \bar \Gamma_{j}(x_{j+1}) + \frac{1}{2\lam} \|x_{j+1}-\hat x_{k-1}\|^2 \right) \\
		& \overset{\eqref{ineq:Gammaj}}{\ge} (1-\tau) [\ell_f(x_{j+1};x_j) + h(x_{j+1})] + \tau \left( m_j + \frac{1}{2 \lam_\mu} \|x_{j+1} -x_j\|^2 \right)
		\end{align*}
		which, together with \eqref{rel:tau}, implies \eqref{ineq:mj}.
	\end{proof}
	
	The next result,
 which
	plays an important role in the analysis
	of the null iterates,
 establishes
	a key recursive formula for the sequence $\{t_j\}$ defined in \eqref{ineq:hpe1}.
	
	\begin{lemma}\label{lem:tj}
	    For every  $j\in {\cal C}_k\setminus \{j_k\}$, we have
     \begin{equation}\label{ineq:tj-recur}
         t_{j+1}-\frac{\bar \varepsilon}4 \le \tau \left(t_j -\frac{\bar \varepsilon}4\right).
     \end{equation}
	\end{lemma}
	
	\begin{proof} 
        Using \eqref{ineq:est} with $(M_f,L_f,x,y)=(\bar M_f, \bar L_f, x_{j+1},x_j)$ and the fact that $\phi=f+h$, we have
        \begin{equation}\label{ineq:ellphi}
            \ell_f(x_{j+1};x_j) + h(x_{j+1}) + \frac{\bar L_f}2 \|x_{j+1}-x_j\|^2\ge \phi(x_{j+1}) - 2 \bar M_f \|x_{j+1}-x_j\|.
        \end{equation}
        This inequality and \eqref{ineq:mj} imply that
		\begin{align}
			m_{j+1} - \tau m_j
			&\overset{\eqref{ineq:mj}}{\ge} (1-\tau) \left[ \ell_f(x_{j+1};x_j) + h(x_{j+1}) +   \left(\frac{\bar L_f}2 +\frac{4\bar M_f^2}{\bar \varepsilon}\right) \|x_{j+1}-x_j\|^2 \right] \nn \\
			&\overset{\eqref{ineq:ellphi}}{\ge} (1-\tau) \phi(x_{j+1}) + \frac{1-\tau}{\bar \varepsilon}\left(4\bar M_f^2\|x_{j+1} -x_j\|^2 - 2\bar M_f \bar \varepsilon\|x_{j+1} -x_j\|\right) \nn \\
			&\ge   (1-\tau) \phi(x_{j+1}) -  \frac{(1-\tau)\bar \varepsilon}{4}, \label{ineq:mj2}
		\end{align}		
		where the last inequality is due to the inequality
		$ a^2-2ab \ge - b^2$ with $a=2\bar M_f\|x_{j+1}-x_j\|$ and $b=\bar \varepsilon/2$.
		Using the above inequality and the definitions of $y_{j+1}$ and $ t_{j+1} $ in \eqref{eq:yj-intuition} and \eqref{ineq:hpe1}, respectively, we conclude that
        \begin{align*}
            t_{j+1} &\overset{\eqref{ineq:hpe1}}{=} \phi(y_{j+1}) - m_{j+1} \overset{\eqref{ineq:mj2}}{\le} \phi(y_{j+1}) -\tau m_j - (1-\tau) \phi(x_{j+1}) + \frac{(1-\tau)\bar \varepsilon}{4} \\
            &\overset{\eqref{ineq:hpe1}}{=} \phi(y_{j+1}) -\tau [\phi(y_j)-t_j] - (1-\tau) \phi(x_{j+1}) + \frac{(1-\tau)\bar \varepsilon}{4} \\
            &\overset{\eqref{eq:yj-intuition}}{\le} \tau t_j + \frac{(1-\tau)\bar \varepsilon}{4},
        \end{align*}
  and that the lemma holds.
	\end{proof}
	

	

	The next lemma gives a uniform bound on $t_{j_{k}+1}$ which is used in Proposition \ref{prop:null-strong} to derive a uniform bound on the maximum number of consecutive null iterates generated by GPB.
	Its proof uses Lemma \ref{lem:consecutive} in Appendix \ref{sec:technical}
	where a crucial bound on $\|x_{j_{k}+1}-\hat x_{k}\| = \|x_{j_{k}+1}-x_{j_{k}}\|$ is obtained.

    \begin{lemma}\label{lem:t1} 
    For every $k\ge 0$, we have $ t_{j_k+1}\le \bar t$
		where	\begin{equation}\label{def:bar t}
		    \bar t:= \bar M_f^2 + 4(\bar L_f+2)(\max\{1,2\lam \bar L_f\}d_0 + \lam \bar M_f)^2.
		\end{equation}
	\end{lemma}
	
	\begin{proof}
	    Using both \eqref{eq:yj-intuition} and \eqref{ineq:hpe1} with $j=j_k+1$, relation \eqref{basic obs ii}, and the facts that $\phi=f+h$ and $ \Gamma_{j_k+1}\ge \ell_f(\cdot;x_{j_k})+h $ (see the serious update in step 1 of GPB), we have
	    \begin{align*}
	        t_{j_k+1} &\overset{\eqref{ineq:hpe1}}{=} \phi(y_{j_k+1}) - m_{j_k+1} \overset{\eqref{eq:yj-intuition}, \eqref{basic obs ii}}{\le} \phi(x_{j_k+1}) - \Gamma_{j_k+1}(x_{j_k+1})
	        \le f(x_{j_k+1}) - \ell_f(x_{j_k+1};x_{j_k}) \\
	        &\overset{\eqref{ineq:est}}{\le} 2\bar M_f \|x_{j_k+1}-x_{j_k}\| + \frac{\bar L_f}{2} \|x_{j_k+1}-x_{j_k}\|^2
	        \le \bar M_f^2 + \left( \frac{\bar L_f}{2}+1\right)\|x_{j_k+1}-x_{j_k}\|^2
	    \end{align*}
		where the third inequality is due to \eqref{ineq:est} with $(M_f, L_f, x, y)=(\bar M_f, \bar L_f, x_{j_k+1}, x_{j_k})$, and the last inequality is due to the fact that $2ab \le a^2+b^2$ for every $a,b \in \R$.
		The conclusion of the lemma now follows from the above inequality and Lemma \ref{lem:consecutive} in Appendix \ref{sec:technical}.
	\end{proof}
    
    We are now ready to present the main result of this subsection where a bound on $|B(\ell_0)|$ is obtained in terms of $\tau$, $\bar t$ and $\bar \varepsilon$.

	\begin{proposition}\label{prop:null-strong} 
		The set ${\cal C}_k$ is finite and 
	    \begin{equation}\label{eq:inner1}
	        |{\cal C}_k| \le \frac{1}{1-\tau} \log\left( \frac{4 \bar t}{\bar \varepsilon}\right) + 1
	    \end{equation}
	    where $\bar t$ is as in \eqref{def:bar t} and $\tau$ is as in step 0 of GPB. In particular, if $\tau$ is as in \eqref{eq:tau}, then
        \begin{equation}\label{eq:inner2}
            |{\cal C}_k| \le \left( 1 + \frac{8\lam_\mu(M_f^2 + \bar \varepsilon L_f)}{\bar \varepsilon} \right) \log\left( \frac{4\bar t}{\bar \varepsilon}\right) + 1.
        \end{equation}
	\end{proposition}
	\begin{proof}
	    Using the inequality $\tau \le e^{\tau-1}$,
	    and  Lemmas \ref{lem:tj} and \ref{lem:t1}, we then conclude that for every $j \in {\cal C}_k$,	
		\[
		t_j - \frac{\bar \varepsilon}{4}\le \tau^{j-j_{k-1}-1}\left(t_{j_{k-1}+1} - \frac{\bar \varepsilon}{4} \right) \le \tau^{j-j_{k-1}-1}t_{j_{k-1}+1} \le e^{(\tau-1) (j-j_{k-1}-1)} \bar t. 
		\]
		Using this observation, and noting that
		step 1 of GPB and
		the definition of ${\cal C}_k$ imply that
		$t_j > \bar \varepsilon/2$
		for every $j \in {\cal C}_k \setminus \{j_k\}$, it is now easy to see that \eqref{eq:inner1} follows.
        Since $\tau$ as in \eqref{eq:tau} satisfies \eqref{rel:tau1}, it immediately follows that \eqref{eq:inner2} holds in view of \eqref{eq:tau} and \eqref{eq:inner1}.
    \end{proof}

	\subsection{The total iteration-complexity of GPB}
	
	This subsection establishes the total iteration-complexity of GPB.
	
	
	We start by providing a more general version
	of Theorem \ref{thm:bound-cvx} which does not impose any
	condition on $\lam$.

	\begin{proposition}\label{prop:total}
	Let $(x_0,\lam, \bar \varepsilon) \in \dom h \times \R_{++} \times \R_{++} $ and $\tau$ as in \eqref{eq:tau} be given.
		Then, any variant of
		GPB with input ($x_0,\lam, \bar \varepsilon,\tau$)
		obtains a $\bar \varepsilon$-solution of \eqref{eq:ProbIntro}
		in a number of iterations bounded by
	\begin{equation}\label{cmplx:total-strong}
		\left[ \left( 1 + \frac{8\lam_\mu(M_f^2 + \bar \varepsilon L_f)}{\bar \varepsilon} \right) \log\left( \frac{4\bar t}{\bar \varepsilon}\right) + 1 \right] \left[ \min \left\lbrace \frac{d_0^2}{\lam \bar \varepsilon}, \frac{1}{\mu \lam_\mu} \log\left(  \frac{\mu d_0^2}{\bar \varepsilon}+1 \right) \right\rbrace  + 1\right]
	\end{equation}
    where $\bar t$ is as in \eqref{lem:t1}.
	\end{proposition}
	
	\begin{proof}
	    This proposition is a direct consequence of Propositions \ref{prop:outer} and \ref{prop:null-strong}.
	\end{proof}
	
	
	Since $\tau$-free GPB instances do not depend on $\tau$, we can choose $\tau$ as in \eqref{eq:tau} with $(M_f,L_f)$ replaced by $(\bar M_f,\bar L_f)$.
	Hence, the $\bar \varepsilon$-iteration complexity for $\tau$-free GPB instances is \eqref{cmplx:total-strong} with $M_f^2 + \bar \varepsilon L_f$ replaced by $T_{\bar \varepsilon}$.
	
	Proposition \ref{prop:total} allows us to make one additional remark about Theorem \ref{thm:bound-cvx},
	namely, in
	the unusual case where the range of $\lam$ \eqref{ineq:lam2} is empty, i.e., $C^2 (M_f^2 + \bar \varepsilon L_f) d_0^2/\bar \varepsilon^2 < 1$, it can be easily seen that \eqref{cmplx:total-strong}, up to a logarithmic term, reduces to
	${\cal O}([\kappa + 1][C^{-2}\kappa^{-1}+1])$
	where $\kappa := \lam (M_f^2 + \bar \varepsilon L_f)/\bar \varepsilon$. Hence, the $\bar \varepsilon$-iteration complexity of GPB with $\lam = \bar \varepsilon/[C (M_f^2 + \bar \varepsilon L_f)]$
	becomes ${\cal O}((1+C^{-1})^2)$, which shows that the instances of \eqref{eq:ProbIntro} for which  \eqref{ineq:lam2} does not hold
	can be trivially solved by GPB with a proper choice of the prox stepsize.

	We are now ready to prove Theorem \ref{thm:bound-cvx}.
	



	
	
	
	\noindent
	{\bf Proof of Theorem \ref{thm:bound-cvx}}
	Defining
	\begin{equation}\label{def:ab}
	    a= \frac{\lam_\mu (M_f^2 + \bar \varepsilon L_f)}{\bar \varepsilon}, \quad 
	b= \min \left\lbrace \frac{d_0^2}{\lam \bar \varepsilon}, 
	\frac{1 }{\mu \lam_\mu} \log\left(  \frac{\mu d_0^2}{\bar \varepsilon} +1 \right) \right\rbrace,
	\end{equation}
	and using \eqref{cmplx:total-strong}, we conclude that
	$ {\cal O}((a+1)(b+1)) $ 
	is a $ \bar \varepsilon $-iteration complexity bound for GPB up to a logarithmic term.
	We break the proof into two cases: 1) $\mu \le C(M_f^2+\bar \varepsilon L_f)/\bar \varepsilon^2$; and 2) $\mu \ge C(M_f^2+\bar \varepsilon L_f)/\bar \varepsilon^2$.
	
	First, assume that case 1 holds. Using the definition of $\lam_\mu$ in \eqref{eq:lam-mu}, the fact that $\mu \le C(M_f^2+\bar \varepsilon L_f)/\bar \varepsilon^2$, and the first inequality in \eqref{ineq:lam1}, we have 
	\begin{equation}\label{ineq:lammu1}
	    \frac 1{\lam_\mu} = \frac 1{\lam} + \mu \le \frac{2C (M_f^2 + \bar \varepsilon L_f)}{\bar \varepsilon},
	\end{equation}
	and hence $ a\ge 1/(2C) $.
	Moreover, it follows from the definition of $b$ in \eqref{def:ab} and the second inequality in \eqref{ineq:lam1} that
	\begin{equation}\label{ineq:b}
		b\ge \min \left\lbrace \frac1C \, ,  \,
		\frac{1 }{\mu \lam_\mu } \log\left(  \frac{\lam \mu}{C} +1 \right) \right\rbrace.    
	\end{equation}
	Using the fact
	that $\log(1+t) \ge t/(1+t)$ for every $t>0$,
	we easily see that
	$\log(1+t) \ge t/2$ if
	$t \le 1$ and
	$\log(1+t) \ge \log 2>0$ if $t \ge 1$. 
	This observation with $t=\lam \mu/C$ and the definition of $\lam_\mu$ in \eqref{eq:lam-mu} then imply that
	\[
	\frac{1}{\mu\lam_\mu} \log\left(  \frac{\lam \mu}{C} +1 \right) 
	\ge \min \left \lbrace \frac{\lam}{2 \lam_\mu C} \, , \, \left( 1+ \frac{1}{\lam \mu} \right) \log2 \right \rbrace 
	\ge \min \left \lbrace \frac{1}{2 C} \, , \, \log2 \right \rbrace,
	\]
	and hence that
	$b \ge \min\{ 1/(2C), \log 2 \}$.	
	This inequality and the fact that $ a\ge 1/(2C) $ imply that
	$ {\cal O}((a+1)(b+1)) $
	is equal to
	${\cal O}(ab+1)$.
	Using this observation, the definitions
	of $a$ and $b$ in \eqref{def:ab}, and the fact that $\lam_\mu \le \lam$,
	we then conclude that
	the bound 
	$ {\cal O}((a+1)(b+1)) $
	reduces to \eqref{eq:bound-mu}, 
	and hence that the theorem holds for case 1.
	
	Assume now that case 2 holds. Then, it follows from the definition of $\lam_\mu$ in \eqref{eq:lam-mu} and the first inequality in \eqref{ineq:lam1} that 
	\begin{equation}\label{ineq:lammu}
	   \frac{1}{\lam_\mu} =
	\mu + \frac{1}{\lam}\ge \mu \ge \frac{C(M_f^2+\bar \varepsilon L_f)}{\bar \varepsilon},
	\quad \lam \mu \ge 1. 
	\end{equation}
	The first inequality then implies that
	$a \le 1/C$ in view of the first identity in \eqref{def:ab}, and hence
	that
	$ {\cal O}((a+1)(b+1)) $ is ${\cal O}(b+1)$. We will now  derive a bound on $b$.
	Indeed, using the definitions of $b$ and $\lam_\mu$ in \eqref{def:ab} and \eqref{eq:lam-mu}, respectively, we have
	\begin{equation}\label{ineq:b2}
	    b=\min \left\lbrace \frac{d_0^2}{\lam \bar \varepsilon}, 
	\left( 1+\frac{1}{\lam \mu} \right) \log\left(  \frac{\mu d_0^2}{\bar \varepsilon} +1 \right) \right\rbrace 
	\le \min \left\lbrace \frac{C(M_f^2 + \bar \varepsilon L_f) d_0^2}{\bar \varepsilon^2} \,,\,
	2 \log\left(  \frac{\mu d_0^2}{\bar \varepsilon} +1 \right) \right\rbrace
	\end{equation}
	where the inequality is due to the second inequality in \eqref{ineq:lammu} and the first inequality in \eqref{ineq:lam1}.
	Hence, the bound ${\cal O}(b+1)$ becomes
	\[
	{\cal O} \left( \min \left\lbrace \frac{(M_f^2 + \bar \varepsilon L_f) d_0^2}{\bar \varepsilon^2}, 
	 \log\left(  \frac{\mu d_0^2}{\bar \varepsilon} +1 \right) \right\rbrace +1 \right).
	\]
	Finally, it is easy to see that bound \eqref{eq:bound-mu} becomes the above bound when $\mu \ge C(M_f^2+\bar \varepsilon L_f)/\bar \varepsilon^2$, and hence that the theorem holds for case 2.
	\QEDA
	
	It is worth pointing out how condition \eqref{ineq:lam1} on the prox stepsize is used in the proof of Theorem~\ref{thm:bound-cvx}.
	Indeed, the first inequality in \eqref{ineq:lam1} is used to obtain the inequality in~\eqref{ineq:lammu1},
	the last inequality in~\eqref{ineq:lammu}, and the inequality in~\eqref{ineq:b2}, 
	while the second inequality in \eqref{ineq:lam1} is used to obtain \eqref{ineq:b}.

\section{A One-Cut Adaptive Proximal Bundle Method}\label{sec:adap}

This section presents an adaptive version of the 1C-PB method, referred to as the 1C-APB method which, in contrast to 1C-PB,
does not require the availability of a triple $(L_f,M_f,\mu)$
satisfying (A1) and (A3), and which has
the same $\bar \varepsilon$-iteration complexity as described in Corollary \ref{cor:E2E3} for an arbitrary
$\tau$-free GPB variant.

We start by stating the 1C-APB method.

	\noindent\rule[0.5ex]{1\columnwidth}{1pt}
	
	1C-APB
	
	\noindent\rule[0.5ex]{1\columnwidth}{1pt}
	\begin{itemize}
		\item [0.] Let $ x_0\in \dom h $, $\lambda>0$, $\beta\ge 1$ and $ \bar \varepsilon>0 $
		be given, 
		and set  
		$y_0=x_0$, $t_0=0$, $\tau_0 =0$, and $j=1$;
	    	\item[1.]  set $\tau=\tau_{j-1}/\beta$;
	    	\item[2.]
	    	if $t_{j-1} \le \bar \varepsilon/2$, then
	    	perform a {\bf serious update}, i.e.,
	    	set $x^c_{j} = x_{j-1}$
			and
			$\Gamma_{j} = \ell_f(\cdot;x_{j-1})+h$;
			else, perform a {\bf null update}, i.e.,
			set $x^c_{j}= x^c_{j-1}$ and $\Gamma_{j} = \tau \Gamma_{j-1} + (1-\tau) [\ell_f(\cdot;x_{j-1})+h]$;
	    \item[3.]
	    compute $x_{j}$, $y_{j}$, $m_{j}$ and $t_{j}$ as in step 2 of GPB;
		\item[4.]  if $t_{j-1} > \bar \varepsilon/2$ and
		$t_{j} > \tau t_{j-1} + (1-\tau) \bar \varepsilon/4$, then set $\tau = (1+\tau)/2$ and go to step 2; else,
		set $\tau_{j}=\tau$ and $j\leftarrow j+1$, and go to step 1.
	\end{itemize}
	\rule[0.5ex]{1\columnwidth}{1pt}

 We use the same terminology (e.g., serious iteration) as defined in the paragraph following GPB.
    For ease of discussion in this subsection, we define $\bar \tau$ as follows
    \begin{equation}\label{eq:tau1}
		\bar \tau:= \left[ 1+\frac{(1+\lam \mu)\bar \varepsilon}{8\lam T_{\bar \varepsilon}^2}\right]^{-1}
	\end{equation}
    where $T_{\bar \varepsilon}$ is as in \eqref{def:T}. We note that $\bar \tau$ is the smallest $\tau\in(0,1)$ satisfying \eqref{rel:tau1}.
     
	We now make some remarks about the 1C-APB method.
	First, in contrast to the GPB framework which does not specify how some quantities are generated, 1C-APB is a well-determined method since it specifies $\Gamma_{j}$ in both the serious and null updates,
	the latter of which
	computes $\Gamma_{j}$ based on
	the one-cut bundle update scheme
	(E1).
	Second,
	the iteration count $j$ is only increased in step 4 and when that happens the key inequality
	\begin{equation}\label{ineq:tau-key}
	    t_{j}-\frac{\bar \varepsilon}4 \le \tau_{j} \left(t_{j-1} -\frac{\bar \varepsilon}4 \right)
	\end{equation}
	is satisfied.
	Before that happens,
 1C-APB can loop a few times between steps 2 and 4 and,
 in the process,
 computes intermediate
 quantities which depends on $\tau$ and (with some abuse of notation) are
 all denoted by $\Gamma_{j}$, $x_{j}$, $y_{j}$, $m_{j}$ and $t_{j}$.
 Third, since $\tau_0=0 < \bar \tau$, it may happen
 that many $\tau_j$'s 
 will also be less than $\bar \tau$. Hence,
 1C-APB can not be viewed as a special case of
 GPB since the latter one
 requires its constant $\tau$ to be at least $\bar \tau$.
Finally,  $\{\tau_j\}$ is a non-decreasing sequence if $\beta=1$ but it can decrease
if $\beta >1$.

	
    The following lemma summarizes some basic properties of 
    1C-APB.

	\begin{lemma}\label{lem:tau}
	   The following statements about the 1C-APB method hold:
	   \begin{itemize}
            \item[a)] 
            $0\le \tau_j \le (1+\bar \tau)/2$ for every $j\ge 0$;
	       \item[b)] for every serious iteration $j_k$, $t_{j_k}\le \bar \varepsilon/2$ and \eqref{ineq:tau-key} holds for every  $j \in {\cal C}_k\setminus \{j_{k-1}+1\}$.
	   \end{itemize}
	\end{lemma}
	
	\begin{proof}
	    a) It follows from Lemma \ref{lem:tj} that if $\tau_j \ge \bar \tau$
	    then $\tau_{\ell} = \tau_j$ for every $\ell > j$.
     This statement now immediately follows from 
	    this observation, the fact that $\tau_0=0$, and the way the
	    sequence $\{\tau_j\}$ is generated.
	    
	    b) This statement follows immediately from steps 2 and 4 of 1C-APB.
	\end{proof}
	
    The following result is similar to Proposition \ref{prop:null-strong} and establishes a bound on the maximum number of consecutive
    null iterates
    generated by 1C-APB. 
    
	\begin{proposition}\label{prop:null-strong2}
		The following statements about
		1C-APB hold:
		\begin{itemize}
		    \item [a)] in each iteration, the number of times $\tau$ is updated in step 4 is at most
		    \begin{equation}\label{eq:log}
		        1 + \left \lceil \log\left( 1+ \frac{8\lam_{\mu} T_{\bar \varepsilon}^2}{\bar \varepsilon}\right) \right \rceil;
		    \end{equation}
		    \item [b)] if $j_{k-1} $ is a serious iteration of the 1C-APB method, then
		    the next serious iteration $ j_k $ happens and satisfies
		    \[
		    j_k-j_{k-1}\le  2\left( 1+ \frac{8\lam_{\mu} T_{\bar \varepsilon}^2}{\bar \varepsilon}\right)  \log\left( \frac{4\bar t}{\bar \varepsilon}\right) + 1
		    \]
		    where $T_{\bar \varepsilon}$, $\lam_{\mu}$, and $\bar t$ are as in \eqref{def:T}, \eqref{eq:lam-mu}, and \eqref{def:bar t}, respectively.
		    
		\end{itemize}
	\end{proposition}
	\begin{proof}
	    a) It follows from the way $\tau$ is updated in step 4 that
	    $1 - \tau^+ = (1-\tau)/2$
	    where $\tau^+$ is the updated $\tau$.
	    Using this observation and Lemma \ref{lem:tau}(a), we then easily conclude that the number of times $\tau$ changes is bounded by $1+\lceil\log\left( 1/(1-\bar \tau) \right) \rceil$.
	    The conclusion in a) now follows from the last conclusion and the definition of $\bar \tau$ in \eqref{eq:tau1}.
	    
	    b) It follows from Lemma \ref{lem:tau} (a) and (b) that for every $j\in {\cal C}_k\setminus \{j_{k-1}+1\}$,
	    \[
	    t_{j}-\frac{\bar \varepsilon}4 \le \frac{1+\bar \tau}{2} \left(t_{j-1} -\frac{\bar \varepsilon}4\right).
	    \]
	    Using the inequality above, the fact that $t_{j_k}\le \bar \varepsilon/2$ (see Lemma \ref{lem:tau}(b)) and Proposition \ref{prop:null-strong}, we conclude that 
	    \[
	    j_k-j_{k-1} \le  \frac{2}{1-\bar \tau} \log\left( \frac{4 \bar t}{\bar \varepsilon}\right) + 1.
	    \]
	    The above inequality, \eqref{eq:lam-mu}, and the definition of $\bar \tau$ in \eqref{eq:tau1} immediately imply b).
    \end{proof}
    

    We now discuss the $\bar \varepsilon$-iteration complexity of 1C-APB.

    \begin{theorem}\label{thm:adaptive}
    Let initial point $x_0 \in \dom h$, tolerance $\bar \varepsilon>0$ and prox stepsize $\lam>0$ be given, and consider an instance $(f,f';h)$ of \eqref{eq:ProbIntro}
    satisfying conditions (A1)-(A3).
    Then, the $\bar \varepsilon$-iteration complexity for 1C-APB is
        \begin{equation}\label{eq:bound-adap}
		 \left[ 2\left( 1 + \frac{8\lam_{\mu} T_{\bar \varepsilon}^2}{\bar \varepsilon} \right) \log\left( \frac{4\bar t}{\bar \varepsilon}\right) + 1 \right] \left[ \min \left\lbrace \frac{d_0^2}{\lam \bar \varepsilon}, \frac{1}{\mu \lam_{\mu}} \log\left(  \frac{\mu d_0^2}{\bar \varepsilon}+1 \right) \right\rbrace  + 1\right].
		\end{equation}
	As a consequence, if in addition the instance $(f,f';h)$
	and the input triple $(x_0,\lam,\bar \varepsilon)$ satisfy  \eqref{ineq:lam2}, then the $\bar \varepsilon$-iteration complexity for 1C-APB is (up to a logarithmic term) given by \eqref{eq:bound-bar}.
    \end{theorem}
    
    \begin{proof}
        First, the same analysis as in Subsection \ref{subsec:serious} shows that the number of serious iterations of 1C-APB is bounded by \eqref{bound:outer}.
        Hence, this conclusion and Proposition \ref{prop:null-strong2}(b) imply that the $\bar \varepsilon$-iteration complexity for 1C-APB is given by \eqref{eq:bound-adap}.
        Letting $a=\lam_{\mu} T_{\bar \varepsilon}^2/\bar \varepsilon$ and $b$ be as in \eqref{def:ab}, and using \eqref{eq:bound-adap}, we have $ {\cal O}((a+1)(b+1)) $ is the $\bar \varepsilon$-iteration complexity for 1C-APB up to a logarithmic term. 
        Using the assumption \eqref{ineq:lam2} and following a similar argument as in the proof of Theorem \ref{thm:bound-cvx}, 
        we conclude that the $\bar \varepsilon$-iteration complexity for 1C-APB is (up to a logarithmic term) given by \eqref{eq:bound-bar}.
    \end{proof}

    It is worth noting that
    a result similar to
    Corollary \ref{cor:holder} 
    dealing with instances
    $(f,f';h)$ of \eqref{eq:ProbIntro} satisfying (A1), (A2), and \eqref{eq:hybrid-holder}
    can also be established
    for 1C-APB.
    
    
    We end this section by discussing the complexity of 1C-APB in terms of the total number of resolvent evaluations of $\partial h$, i.e., an evaluation of the point-to-point operator $(I+\alpha \partial h)^{-1}(\cdot)$ for some $\alpha>0$.
    Observe first that the computation of
    $x_{j}$ in step 3 of 1C-APB requires one resolvent evaluation of $\partial h$ due to \eqref{def:xj} and the fact that
    $\Gamma_{j}$ has the form \eqref{eq:Gamma-form}. Hence, 
    the total number of resolvent evaluations of $\partial h$ is
    bounded by the number that step 3 is performed.
    Thus, it follows from
    Theorem \ref{thm:adaptive} and  Proposition \ref{prop:null-strong2}(a) that the total number of resolvent evaluations of $\partial h$ is bounded by the product of
    \eqref{eq:log} and \eqref{eq:bound-adap} if $\beta>1$ or the sum of \eqref{eq:log} and \eqref{eq:bound-adap} if $\beta=1$.

    \section{Concluding Remarks}\label{sec:conclusion}

	This paper presents a generic proximal bundle framework, namely, GPB, for solving the HCCO problem \eqref{eq:ProbIntro}. 
	Instead of focusing on a specific bundle update scheme, GPB is based on a generic one, i.e., the BU blackbox, which includes three schemes, namely, multiple-cuts (E3), two-cuts (E2), and a novel one-cut scheme (E1).
	Moreover, this paper considers the hybrid case where (A3) holds and presents a unified and simple analysis for GPB.
	It establishes two $\bar \varepsilon$-iteration complexity for GPB instances, namely, \eqref{eq:bound-mu} for 1C-PB and \eqref{eq:bound-bar} for the $\tau$-free GPB instances (i.e., 2C-PB and MC-PB). 
	Finally, this paper presents the 1C-APB method which is an adaptive version of 1C-PB and shows that 1C-APB has the same $\bar \varepsilon$-iteration complexity as the $\tau$-free GPB instances. 
	
	We briefly discuss the relationship between GPB instances and other methods.
	First, the CS-CS method can be viewed as a special instance of any GPB variant with a relatively small prox stepsize.
 Second, it is worth noting that 1C-PB  has slight similarity with the dual averaging (DA) method of \cite{nesterov2009primal} since both methods explore the idea of aggregating cuts into a single one.
	However, there are essential differences between the two methods:
    1) DA uses variable prox stepsizes, while 1C-PB uses a constant one;
    and 2) most importantly,
    1C-PB updates the prox-center immediately after every serious iteration, while
    DA uses a fixed prox-center throughout the process.
	


	We finally discuss some possible extensions of our analysis in this paper.
	
 First, under the assumption that
 the diameter $D$ of $\dom h$ is finite,
 it follows from the last inequality in Subsection 3.1 of \cite{lan2012optimal} that
the $\bar \varepsilon$-iteration complexity of
an accelerated composite subgradient method proposed in \cite{lan2012optimal} is 
\[
{\cal O}\left(\frac{\sqrt{L_f} D}{\sqrt{\bar \varepsilon}} + \frac{M_f^2 D^2}{\bar \varepsilon^2}\right). 
\]
Moreover, it follows from the Introduction of
\cite{lan2012optimal} (see the paragraph containing equation (6) there) that the
above bound is optimal for the HCCO problem class determined by
$L_f$, $M_f$ and $D$.
 In this regards, the $\bar \varepsilon$-iteration complexity of GPB is optimal 
 when $L_f=0$ (i.e., in the pure nonsmooth case), but it is not optimal when $L_f>0$. It would be interesting to design an accelerated variant of GPB which  
	is optimal for the aforementioned
 HCCO problem class.
 
	Second, proximal bundle methods have not been studied in the context of stochastic subgradient oracles with continuous distribution, and hence it is interesting to investigate such methods by using the techniques developed in this paper.
 
 Third, a drawback of GPB is that its cycle termination criterion, namely,
 $t_j \le \bar \varepsilon/2$, depends on the tolerance $\bar \varepsilon$ specified for it. An interesting question is whether it is possible to develop a variant of GPB with a cycle termination criterion which
 does not depend on the tolerance
 $\bar \varepsilon$.

Finally, we address issues related to the strongly convex case
(i.e., $\mu>0$).
Our analysis
assumes that
$f$ is convex and $h$ is $\mu$-convex and shows that
(see Theorem \ref{thm:bound-cvx}),
even though GPB does not
require $\mu>0$,
the dependence of its iteration-complexity bound 
\eqref{eq:bound-mu} on $\mu$
and $\bar \varepsilon$ is (up to a logarithmic term) the same as
that for the CS-CS method
(see Proposition \ref{prop:sub-new}).
An interesting question is whether GPB or a related
variant which does not require $\mu$ either, directly applied to the HCCO problem \eqref{eq:ProbIntro}
also
has the above iteration-complexity bound
under  the assumption that
 $f$ is $\mu_f$-convex, $h$ is $\mu_h$-convex and $\mu=\mu_f+\mu_h$.
 
 We now mention some papers and observations related
 to the topic of  the previous paragraph.
 Under the assumption that $L_f=0$ and $\mu_h=0$ (and hence, $\mu=\mu_f$),
 the proximal bundle method
 of \cite{du2017rate} is shown to have an
  $\tilde {\cal O}(M_f^2/(\mu \bar \varepsilon))$ iteration-complexity bound
  but requires $\mu_f$ as input
  since its prox stepsize is
  chosen as $\lam =1/\mu_f$.
  Moreover, for the same method and under the same assumptions,
  \cite{diaz2021optimal}  improves the latter bound to
  ${\cal O}(M_f^2/(\mu \bar \varepsilon))$ by removing a logarithmic term.
  Finally, if $\mu_f$ is known
  and the new composite
  structure $(\tilde f,\tilde h)$ defined as
  $\tilde f=f-\mu_f\|\cdot\|^2/2$ and
  $\tilde h=h+\mu_f\|\cdot\|^2/2$
  is considered in place of
  $(f,h)$,
  then GPB
  with this composite structure has iteration-complexity equal
  to \eqref{eq:bound-mu}
  where $\mu=\mu_f+\mu_h$.

	
	
	\bibliographystyle{plain}
	\bibliography{Proxacc_ref}
	
	\appendix

	\section{Technical Results}\label{sec:technical}
	
	The main result of this section is Lemma \ref{lem:consecutive} which was used in the proof of Lemma \ref{lem:t1}. It also presents the proof of Proposition \ref{lem:hybrid}.
	
	Before stating and proving Lemma \ref{lem:consecutive},
	we first present two technical results.
	
	
	
	\begin{lemma}\label{lem:prox-dist}
	Let $x\in \R^n$, $0<\tilde \lam < \lam$ and $\Gamma\in \bConv{n}$ be given, and
	define
	\[
	    x^+=\underset{u\in \R^n}\argmin \left\lbrace \Gamma(u) +\frac{1}{2 \lam}\|u- x \|^2 \right\rbrace,\quad
	    \tx^+=\underset{u\in \R^n}\argmin \left\lbrace \Gamma(u) +\frac{1}{2 \tilde \lam}\|u-x\|^2 \right\rbrace.
	\]
	Then, we have $\|x^+-x\|\le (\lam/\tilde \lam) \|\tx^+-x\| $.
	\end{lemma}
	
	\begin{proof}
	    Denote $\partial \Gamma$ by $A$, and define
	    \[
	    y_{A}(\lam; x):=(I+\lam A)^{-1}(x), \quad \varphi_{A}(\lam; x):=\lam\|y_{A}(\lam;x)-x\|.
	    \]
	    It is easy to see that 
	    \[
	    \|x^+-x\| = \|y_{A}(\lam;x)-x\|= \frac1\lam \varphi_{A}(\lam;x),
	    \quad
	    \|\tx^+-x\| = \|y_{A}(\tilde\lam;x)-x\|= \frac1{\tilde\lam} \varphi_{A}(\tilde\lam;x).
	    \]
	    The conclusion of the lemma now follows  from the above observation and the second inequality in (39) of \cite{Newton}
	    which claims 
	    that
	    \[
	    \varphi_{A}(\lam;x) \le \frac{\lam^2}{\tilde \lam^2} \varphi_{A}(\tilde\lam;x).
	    \]
	\end{proof}
	
	
	\begin{lemma}\label{lem:descent}
	Let $(\Gamma,z_0,\lam)\in \mConv{n} \times \R^n \times (0,1/\bar L_f)$ be a triple such that 
 \begin{equation}\label{ineq:local}
     \ell_f(\cdot;z_0)+h \le \Gamma \le \phi
 \end{equation}
 and
	define
	\begin{equation}\label{eq:x}
	    z:=\underset{u\in \R^n}\argmin \left\lbrace \Gamma(u) +\frac{1}{2 \lam}\|u- z_0 \|^2 \right\rbrace.
	\end{equation}
	Then, for every $u\in \dom h$, we have
	\begin{equation}\label{ineq:descent}
	    \frac{1}{2}\left( \mu + \frac{1}{\lam}\right) \|u-z\|^2 + \phi(z) - \phi(u) \le \frac{1}{2\lam} \|u-z_0\|^2 + \frac{2\lam \bar M_f^2}{1-\lam \bar L_f}.
	\end{equation}
	\end{lemma}
	
	\begin{proof}
	    It follows from the assumption that $\Gamma\in \mConv{n}$ that the function $\Gamma+\|\cdot-z_0\|^2/(2\lam)$ is $(\mu + \lam^{-1})$-strongly convex.
	    This conclusion, \eqref{ineq:local}, \eqref{eq:x} and Theorem 5.25(b) of \cite{beck2017first} with $ f= \Gamma+\|\cdot-z_0\|^2/(2\lam) $, $ x^*=z $ and $ \sigma= \mu + \lam^{-1} $, then imply that for every $u\in \dom h$,
	    \begin{align*}
	    \phi(u) + \frac{1}{2\lam}\|u-z_0\|^2 &\overset{\eqref{ineq:local}}{\ge} \Gamma(u) + \frac{1}{2\lam}\|u-z_0\|^2\\
	    &\overset{\eqref{eq:x}}{\ge} \Gamma(z) + \frac{1}{2\lam}\|z-z_0\|^2 + \frac12 \left(\mu + \frac{1}{ \lam}\right) \|u-z\|^2\\
	    &\overset{\eqref{ineq:local}}{\ge} \ell_f(z;z_0) + h(z) + \frac{1}{2\lam}\|z-z_0\|^2+ \frac12 \left(\mu + \frac{1}{ \lam}\right) \|u-z\|^2.
	    \end{align*}
	    The above inequality, the fact that $\phi=f+h$ and \eqref{ineq:est} with  $(M_f, L_f, x, y)=(\bar M_f, \bar L_f, z,z_0)$ then imply that
	    \begin{align*}
			\frac{1}{2}\left( \mu + \frac{1}{\lam}\right) & \|u-z\|^2 + \phi(z) - \phi(u) 
			\le \frac{1}{2\lam} \|u-z_0\|^2  +\phi(z) - \ell_f(z;z_0) - h(z) - \frac{1}{2\lam} \|z-z_0\|^2 \\
			&\overset{\eqref{ineq:est}}{\le} \frac{1}{2\lam} \|u-z_0\|^2  + 2 \bar M_f\|z-z_0\| - \frac{1-\lam \bar L_f}{2\lam} \|z-z_0\|^2.
		\end{align*}
		The lemma now follows from the above inequality, the fact that $\lam \bar L_f <1$ and the inequality $2ab-a^2\le b^2$ with $a^2=(1-\lam \bar L_f)\|z-z_0\|^2/(2\lam)$ and $b^2=2\lam \bar M_f^2/(1-\lam \bar L_f)$.
	\end{proof}
	
	We are now ready to prove the main technical result of this section which provides a bound on the distance between a serious iterate generated by GPB
	and its consecutive (possibly null or serious) iterate.
	It is worth noting that this result is quite general and
	makes no use of the generic bundle update scheme of Subsection \ref{subsec:update}
	since the step from $x_{\ell_0}$ to $x_{\ell_0+1}$ does not use this update.
	
	\begin{lemma}\label{lem:consecutive}
	If $\ell_0$ is a serious iteration, then
	 \begin{equation}\label{ineq:consecutive}
	    \|x_{\ell_0}-x_{\ell_0+1}\| \le 2\sqrt{2} (\max\{1,2\lam \bar L_f\}d_0 + \lam \bar M_f). 
	 \end{equation}
	\end{lemma}
	
	\begin{proof}
	    For the sake of this proof only,
	    we define the auxiliary stepsize $\tilde\lam:=\min\{\lam,1/(2\bar L_f)\}$ and auxiliary point
	    \[
	        w_{\ell_0}:=\underset{u\in \R^n}\argmin
	    \left\lbrace \Gamma_{\ell_0+1}(u) +\frac{1}{2\tilde \lam}\|u- x_{\ell_0} \|^2 \right\rbrace.
	    \]
	    Since $j=\ell_0$ is a serious index,
	    it follows from step 1 of GPB that
	    $\Gamma_{\ell_0+1}\ge \ell_f(\cdot;x_{\ell_0})+h$, and
	    hence that
	    $(\Gamma,z_0,\lam)=(\Gamma_{\ell_0+1}, x_{\ell_0}, \tilde \lam)$
	    and $z=w_{\ell_0}$ satisfy the assumptions of
	    Lemma \ref{lem:descent}.
	    The conclusion of Lemma \ref{lem:descent} with $(u,z,z_0,\lam)=(x_0^*,w_{\ell_0},x_{\ell_0},\tilde \lam)$ and
	    the fact that $\tilde \lam \le 1/(2\bar L_f)$ then imply
	    that
	    \begin{align*}
	    \frac{1}{2}\left( \mu + \frac{1}{\tilde \lam}\right) \|x_0^*-w_{\ell_0}\|^2 + \phi(w_{\ell_0}) - \phi(x_0^*)
	    \le \frac{1}{2\tilde \lam} \|x_0^*-x_{\ell_0}\|^2 + 4\tilde \lam \bar M_f^2
	    \end{align*}
	    which in turn, in view of
	    the facts that $\phi(w_{\ell_0})\ge \phi^*= \phi(x_0^*)$ and $\mu \ge 0$,
	    and the inequality $(a+b)^{1/2} \le a^{1/2}+ b^{1/2}$ for any $a, b \ge 0$,
	    yields
        \[
        \|x_0^*-w_{\ell_0}\| \le \|x_{\ell_0}-x_0^*\| + 2\sqrt{2} \tilde \lam \bar M_f.
        \]
        This inequality and the triangle inequality 
        then imply that
        \begin{equation}\label{eq:bound-tx}
		    \|x_{\ell_0}-w_{\ell_0}\| \le \|x_{\ell_0}-x_0^*\| + \|x_0^*-w_{\ell_0}\| \le 2\|x_{\ell_0}-x_0^*\| + 2\sqrt{2}\tilde\lam \bar M_f
		    \le 2\sqrt{2} ( d_0 + \tilde\lam \bar M_f)
		\end{equation}
		where the last inequality is due to
		\eqref{ineq:dist} and the fact that $x_{\ell_0}$ is equal to one
		of serious iterates $\hat x_k$
		preceding the last one	generated by GPB.
		On the other hand,
		since $0< \tilde \lam < \lam$ and $\Gamma_{\ell_0+1}\in \bConv{n}$, it follows from Lemma \ref{lem:prox-dist} with $(\Gamma,x)=(\Gamma_{\ell_0+1},x_{\ell_0})$ that
		\[
		\|x_{\ell_0+1}-x_{\ell_0}\| \le \frac{\lam}{\tilde \lam}\|w_{\ell_0}-x_{\ell_0}\|.
		\]
		This inequality together with \eqref{eq:bound-tx} and the fact that $\lam/\tilde \lam = \max\{1,2\lam \bar L_f\}$ clearly implies \eqref{ineq:consecutive}.
	\end{proof}

    We end this section by providing the proof of Proposition \ref{lem:hybrid}.
    
    \noindent
	{\bf Proof of Proposition \ref{lem:hybrid}}
    Using Young's inequality 
\[
ab \le \frac{a^p}{p} + \frac{b^q}{q}
\]
with 
\[
a=\|x-y\|^\nu \left(\frac{1-\nu}{\alpha}\right)^{1-\nu},\quad
b = \left(\frac{\alpha}{1-\nu}\right)^{1-\nu},\quad
p=\frac{1}{\nu},\quad
q=\frac{1}{1-\nu},
\]
where $\alpha>0$ is arbitrary, 
we have
\[
\| x-y\|^\nu \le  \nu\left(\frac{1-\nu}{\alpha}\right)^{\frac{1-\nu}{\nu}} \|x-y\| + \alpha.
\]
It follows from \eqref{eq:hybrid-holder} and the above inequality that
\[
    \|f'(x)-f'(y)\|
    \le 2 M_\nu + L_\nu \alpha + L_\nu \nu\left(\frac{1-\nu}{\alpha}\right)^{\frac{1-\nu}{\nu}} \|x-y\|,
\]
and hence that (A3) holds with $(M_f,L_f)=(M_f(\alpha),L_f(\alpha))$ in view of \eqref{eq:ML}.
Moreover, using \eqref{eq:ML} and the fact that $(a+b)^2\le 2a^2+2b^2$ for every $a, b \in \R$, we have
\begin{align*}
        \inf_{\alpha>0} \{ M_f(\alpha)^2 + \bar \varepsilon L_f(\alpha) \} &\le \min_{\alpha>0} \left\lbrace 2M_\nu^2 + \frac{L_\nu^2 \alpha^2}2+ \bar \varepsilon L_\nu \nu\left(\frac{1-\nu}{\alpha}\right)^{\frac{1-\nu}{\nu}}\right\rbrace \\
    &= 2M_\nu^2+ \bar \varepsilon^{\frac{2\nu}{\nu+1}} L_\nu^{\frac{2}{\nu + 1}} \left[\frac12 (1-\nu)^{\frac{2}{\nu + 1}} + \nu (1-\nu)^\frac{1-\nu}{\nu+1} \right] \\
    &\le 2M_\nu^2+ 2\bar \varepsilon^{\frac{2\nu}{\nu+1}} L_\nu^{\frac{2}{\nu + 1}}, 
    \end{align*}
where the minimization problem is minimized at 
    \[
\alpha=\left(\frac{\bar \varepsilon}{L_\nu}\right)^{\frac{\nu}{\nu+1}}(1-\nu)^{\frac{1}{\nu + 1}}
\]
and the second inequality is due to the fact that $\nu\in (0,1)$.
Hence, \eqref{ineq:inf} holds.
Finally, \eqref{ineq:T} immediately follows from the definition of $T_{\bar \varepsilon}$ in \eqref{def:T},  \eqref{ineq:inf}, and the fact that $\sqrt{a+b}\le \sqrt{a} + \sqrt{b}$ for every $a, b \in \R_{++}$.
    \QEDA

	\section{Useful recursive formulas}\label{sec:recursive}
	
	The following two technical results play important roles in the complexity analysis of both GPB and CS-CS. We start by stating the following simple result for general sequences of nonnegative scalars.
	
	\begin{lemma} \label{lm:easyrecur1}
		Assume that sequences of nonnegative scalars $\{\theta_j\}$, $\{\delta_j\}$, $\{\eta_j\}$ and  $\{\alpha_j\}$ satisfy for every $j\ge 1$, $\theta_j \ge 1$, $\delta_j>0$ and 
		\beq \label{eq:easyrecur1}
		\eta_j \le \alpha_{j-1} - \theta_j \alpha_j + \delta_j.
		\eeq
		Let $\Theta_0:=1$ and $\Theta_j := \Pi_{i=1}^j \theta_i$ for every $j\ge 1$, then we have for every $k\ge 1$,
		\[
		\sum_{j=1}^k \Theta_{j-1} \eta_j
		\le
		\alpha_0 - \Theta_k \alpha_k +
		\sum_{j=1}^k  \Theta_{j-1} \delta_j.
		\]
	\end{lemma}
	
	\begin{proof}
	    Multiplying \eqref{eq:easyrecur1} by $\Theta_{j-1}$ and summing the resulting inequality from $ j=1 $ to $ k $, we have
		\[
			\sum_{j=1}^k \Theta_{j-1} \eta_j \le  \sum_{j=1}^k \Theta_{j-1} \left(  \alpha_{j-1} - \theta_j \alpha_j + \delta_j \right) =  \alpha_0 - \Theta_k \alpha_k +
			\sum_{j=1}^k \Theta_{j-1} \delta_j.
		\]
		Hence, the lemma holds.
	\end{proof}
	
	The next result discusses a special case of the previous lemma in which $\theta_j=\theta$ and $\delta_j=\delta$ for every $j\ge 1$.
	
	\begin{corollary} \label{lm:easyrecur}
		Assume that scalars $\theta \ge 1$ and $\delta>0$, and
		sequences of nonnegative scalars
		$\{\eta_j\}$ and  $\{\alpha_j\}$ satisfy
		\beq \label{eq:easyrecur}
		\eta_j \le \alpha_{j-1} - \theta \alpha_j + \delta \quad \forall j \ge 1.
		\eeq
		Then, the following statements hold:
		\begin{itemize}
			\item[a)]  $ \min_{1\le j \le k} \eta_j \le 2\delta $ for every $ k\ge 1$ such that
			\[
			k \ge \min \left\lbrace \frac{\alpha_0}{\delta},  \frac{\theta}{\theta-1} \log\left( \frac{\alpha_0(\theta-1)}{\delta} + 1 \right)  \right\rbrace
			\]
			with the convention that the second term is equal to the first term
			when $\theta=1$ (Note  that the second term converges to the first term
			as $\theta \downarrow 1$.);
			
			\item[b)] $ \alpha_k \le \alpha_0 + k \delta$ for every $ k\ge 1 $.
		\end{itemize}
	\end{corollary}
	
	\begin{proof}
	a) It follows from Lemma \ref{lm:easyrecur1} with $\theta_j=\theta$ and $\delta_j=\delta$ for every $j\ge 1$ that
	\begin{equation}\label{ineq:conse}
			\sum_{j=1}^k \theta^{j-1} \left[ \min_{1\le j \le k}  \eta_j \right] \le \sum_{j=1}^k \theta^{j-1} \eta_j \le \alpha_0 - \theta^k \alpha_k +
			\sum_{j=1}^k \theta^{j-1} \delta.
		\end{equation}
	Using the fact that $ \theta \ge e^{(\theta-1)/\theta} $
		for every $\theta \ge 1$, 
		we have
		\[
		\sum_{j=1}^{k} \theta^{j-1} = \max \left\lbrace k, \frac{\theta^k-1}{\theta-1} \right\rbrace \ge \max\left\lbrace k, \frac{e^{(\theta-1) k /\theta} - 1}{\theta-1}\right\rbrace.
		\]
		This inequality, \eqref{ineq:conse} and the fact that $\alpha_k\ge 0$ imply that for every $ k\ge 1 $,
		\[
		\min_{1\le j \le k}  \eta_j \le \alpha_0\min \left\lbrace \frac 1k, \frac{\theta-1}{e^{(\theta-1) k /\theta} - 1} \right\rbrace + \delta,
		\]
		which can be easily seen to imply a).
		
		b) This statement follows from \eqref{ineq:conse}, the fact that $\eta_j\ge 0$, and the assumption that $ \theta\ge 1$.
	\end{proof}

	\section{The Composite Subgradient Method}\label{sec:CS pf}

	This section contains two subsections. The first one provides the analysis of the CS-CS method, which is used to derive the $\bar \varepsilon$-iteration complexity of CS-CS in Subsection \ref{subsec:cs}. The second one presents an adaptive variant of CS-CS and establishes the $\bar \varepsilon$-iteration complexity of it. 
	
	\subsection{Analysis of CS-CS}
	
	\begin{proposition}\label{prop:sub-new}
		Let an initial point $x_0\in \dom h$, $(L_f,M_f) \in \R_+^2 $ 
		and instance $(f,f';h)$ satisfying
		conditions (A1)-(A3)  be given.
		Then, the number of iterations performed by CS-CS$(x_0,\lam)$
		with $\lam\le \bar \varepsilon/[4(M_f^2+\bar \varepsilon L_f)] $
		until it finds
		a $ \bar \varepsilon $-solution is bounded by 
		\[
		\left \lfloor \min \left\lbrace \frac{d_0^2}{\lam \bar \varepsilon}, \frac{1+\lam \mu}{\lam \mu} \log \left( \frac{\mu d_0^2}{\bar \varepsilon} + 1 \right) \right\rbrace \right \rfloor +1.
		\]
	\end{proposition}
	\begin{proof} 
		Recall that an iteration of CS-CS$(x_0,\lam)$ is as in \eqref{eq:sub}.
		Noting that \eqref{eq:sub} satisfies \eqref{eq:x} with $(z_0,z,\Gamma)=(x_{j},x_{j+1},\ell_f(\cdot;x_{j})+h)$, and using the facts that
		$\lam\le \bar \varepsilon/[4 (M_f^2+\bar \varepsilon L_f)] \le \bar \varepsilon/(4 T_{\bar \varepsilon}^2) < 1/\bar L_f$, $\ell_f(\cdot;x_{j}) + h \in \mConv{n}$, and $\ell_f(\cdot;x_{j}) + h\le \phi$, we conclude that the assumptions of Lemma \ref{lem:descent} is satisfied.
		Hence, it follows from \eqref{ineq:descent} with $(u,z,z_0)=(x_0^*,x_{j+1},x_{j})$ that
	    \[
	    \phi(x_{j+1}) - \phi^* - \frac{1}{2\lam} \|x_0^*-x_{j}\|^2 + \frac{1+\lam \mu}{2\lam} \|x_0^*-x_{j+1}\|^2
	    \le \frac{2\lam \bar M_f^2}{1-\lam \bar L_f} 
	    \le \frac{\bar \varepsilon}{2}
	    \]
	    where the last inequality is due to the facts that $2\lam \bar M_f^2/(1-\lam \bar L_f)$ is an increasing function in $\lam$ and
	    $\lam\le \bar \varepsilon/(4 T_{\bar \varepsilon}^2)$.
		Since the above inequality with $j=j-1$ satisfies \eqref{eq:easyrecur} with 
		\[
		\eta_j=\phi(x_j) - \phi^*, \quad \alpha_j=\frac{1}{2\lam}\|x_j-x_0^*\|^2, \quad \theta=1+\lam \mu, \quad \delta= \frac{\bar \varepsilon}2 ,
		\]
		it follows from Corollary \ref{lm:easyrecur}(a) and the fact that $ \alpha_0=d_0^2/(2\lam) $ that
		$ \min_{1\le j \le k} \phi(x_j)-\phi^*\le \bar \varepsilon $ for 
		every index $k\ge 1$ such that 
		\[
		k\ge \min \left\lbrace \frac{d_0^2}{\lam \bar \varepsilon}, \frac{1+\lam \mu}{\lam \mu}\log\left( \frac{\mu d_0^2}{\bar \varepsilon} + 1\right)\right\rbrace,
		\]
		and hence that the lemma holds.
	\end{proof}
	
	\subsection{An adaptive CS method}\label{subsec:A-CS}
	
	This subsection present an adaptive variant of the CS-CS method, namely, the A-CS method, and establish $\bar \varepsilon$-iteration complexity of the adaptive method. The proposed method is a universal method for solving the HCCO problem \eqref{eq:ProbIntro} since it does not rely on any problem parameters.
	
	\noindent\rule[0.5ex]{1\columnwidth}{1pt}
	
	A-CS
	
	\noindent\rule[0.5ex]{1\columnwidth}{1pt}
	\begin{itemize}
		\item [0.] Let $ x_0\in \dom h $, $\lam_0 >0$ and $ \bar \varepsilon>0 $
		be given, and set $\lam=\lam_0$ and $j=0$;
	    \item[1.] compute 
	    \[
		x =\underset{u\in  \R^n}\argmin
		\left\lbrace  \ell_f(u;x_{j})+h(u) +\frac{1}{2\lam} \|u- x_{j} \|^2 \right\rbrace;
	    \]
		\item[2.]  {\bf if} $f(x)-\ell_f(x;x_{j}) - \|x-x_{j}\|^2/(2\lam) > \bar \varepsilon/2$, {\bf then} set $\lam = \lam/2$ and go to step 1; {\bf else}, go to
		step 3;
		\item[3.]
		set $\lam_{j+1}=\lam$, $x_{j+1}=x$ and $j\leftarrow j+1$, and go to step 1.
	\end{itemize}
	\rule[0.5ex]{1\columnwidth}{1pt}

    \begin{lemma}\label{lem:ACSCS}
    The following statements hold for A-CS($\lam_0, \bar \varepsilon$):
    \begin{itemize}
        \item[a)] for every $j\ge 0$, we have
        \begin{align}
		&x_{j+1} =\underset{u\in  \R^n}\argmin
		\left\lbrace  \ell_f(u;x_{j})+h(u) +\frac{1}{2\lam_{j+1}} \|u- x_{j} \|^2 \right\rbrace, \label{eq:sub1} \\
		&f(x_{j+1}) -\ell_f(x_{j+1};x_{j}) - \frac{1}{2\lam_{j+1}}\|x_{j+1}-x_{j}\|^2 \le \frac{\bar \varepsilon}2; \label{ineq:iteration}
	    \end{align}
        \item[b)] if $\lam_{j}\le \bar \varepsilon/(4 T_{\bar \varepsilon}^2)$ where $T_{\bar \varepsilon}$ is as in \eqref{def:T}, then \eqref{ineq:iteration} holds with $\lam_{j+1}=\lam_j$;
        \item[c)] $\{\lam_j\}$ is a non-increasing sequence;
        \item[d)] for every $j\ge 0$, 
        \begin{equation}\label{ineq:lamj}
        \lam_j \ge \underline{\lam}:= \min \left \lbrace \frac{\bar \varepsilon}{8T_{\bar \varepsilon}^2}, \lam_0 \right\rbrace.
        \end{equation}
    \end{itemize}
    \end{lemma}
    
    \begin{proof}
    a) This statement directly follows from the description of A-CS.
    
    b) Using \eqref{ineq:est} with $(M_f, L_f, x,y)=(\bar M_f, \bar L_f, x_{j+1},x_{j})$ and the inequality that $a^2+b^2\ge 2ab$ for $a,b \in \R$, we have
    \begin{align*}
    f(x_{j+1})-\ell_f(x_{j+1};x_{j}) - \frac{1}{2\lam_{j}}\|x_{j+1}-x_{j}\|^2 
    &\le 2\bar M_f \|x_{j+1}-x_{j}\| - \frac{1-\lam_{j} \bar L_f}{2\lam_{j}} \|x_{j+1}-x_{j}\|^2\\
    &\le \frac{2\lam_{j} \bar M_f^2}{1-\lam_{j} \bar L_f}
    \le \frac{\bar \varepsilon}2
    \end{align*}
    where the last inequality is due to the assumption that $\lam_{j}\le \bar \varepsilon/(4 T_{\bar \varepsilon}^2)$.
    Hence, \eqref{ineq:iteration} holds with $\lam_{j+1}=\lam_j$.
    
    c) This statement clearly follows from steps 2 and 3 of A-CS.
    
    d) This statement follows trivially from b) and c), and the way $\lam$ is updated in step 2.
    \end{proof}

    \begin{proposition}\label{prop:A-CS}
    Let an initial point $x_0$ and a universal constant $C>0$ be given, and consider an instance $(f,f';h)$ of \eqref{eq:ProbIntro} satisfying conditions (A1)-(A3).
    Moreover, assume $(\lam_0, \bar \varepsilon)\in \R_{++}^2$ is such that $\lam_0\ge \bar \varepsilon/(C T_{\bar \varepsilon}^2)$ where $T_{\bar \varepsilon}$ is as in \eqref{def:T}.
    Then, the following statements hold:
    \begin{itemize}
        \item[a)] A-CS$(\lam_0, \bar \varepsilon)$ has $ \bar \varepsilon $-iteration complexity given by \eqref{eq:bound-bar};
        \item[b)] the total number of times  $\lam$ is halved in step 2 is bounded by
        \[
        \left \lceil \log \left( \max\left\lbrace \frac{8\lam_0 T_{\bar \varepsilon}^2}{\bar \varepsilon}, 1 \right \rbrace \right) \right \rceil.
        \]
    \end{itemize}
    \end{proposition}
    
    \begin{proof}
    a) It follows from the fact that $h$ is $\mu$-convex that the objective function in \eqref{eq:sub1} is $(\mu + \lam_{j+1}^{-1})$-strongly convex.
	Using this conclusion, \eqref{eq:sub1} and Theorem 5.25(b) of \cite{beck2017first},
    we have for every $u\in \dom h$,
    \begin{align*}
        \ell_f(x_{j+1};x_{j}) &+ h(x_{j+1}) + \frac{1}{2\lam_{j+1}} \|x_{j+1}-x_{j}\|^2 + \frac{1}{2}\left( \mu + \frac{1}{\lam_{j+1}}\right) \|u-x_{j+1}\|^2\\
		&\le \ell_f(u;x_{j}) + h(u) + \frac{1}{2\lam_{j+1}} \|u-x_{j}\|^2  \le \phi(u) + \frac{1}{2\lam_{j+1}} \|u-x_{j}\|^2.
    \end{align*}
    It follows from the above inequality with $u=x_0^*$ and \eqref{eq:sub1} that
    \begin{align*}
       (1 + \lam_{j+1} \mu) \|x_0^*-x_{j+1}\|^2 &+ 2\lam_{j+1}[\phi(x_{j+1}) - \phi^*]
		- \|x_0^*-x_{j}\|^2 \\
		&\le 2\lam_{j+1} \left[ f(x_{j+1}) - \ell_f(x_{j+1};x_{j}) - \frac{1}{2\lam_{j+1}} \|x_{j+1}-x_{j}\|^2 \right]
		\le \bar \varepsilon \lam_{j+1}.
    \end{align*}
   Since the above inequality with $j=j-1$ satisfies \eqref{eq:easyrecur1} with 
		\[
		\eta_j=2\lam_{j}[\phi(x_j) - \phi^*], \quad \alpha_j=\|x_j-x_0^*\|^2, \quad \theta_j=1+\lam_{j} \mu, \quad \delta_j= \bar \varepsilon \lam_{j},
		\]
		it follows from Lemma \ref{lm:easyrecur1} and the fact that $ \alpha_0=d_0^2 $ that
		\begin{equation}\label{ineq:Theta}
		    \left(\sum_{j=1}^k  2\lam_{j}\Theta_{j-1} \right) \min_{1\le j \le k} [\phi(x_j) - \phi^*]
		\le \sum_{j=1}^k  2\lam_{j}\Theta_{j-1} [\phi(x_j) - \phi^*]
		\le
		d_0^2 +
		\left( \sum_{j=1}^k  2\lam_{j}\Theta_{j-1} \right) \frac{\bar \varepsilon}{2}
		\end{equation}
		where $\Theta_j = \Pi_{i=1}^j (1+\lam_{j}\mu)$ for every $j\ge 1$.
		Note that it follows from Lemma \ref{lem:ACSCS}(d) that $\Theta_j\ge (1+\underline{\lam}\mu)^j $ for every $j\ge 1$.
		Using this observation, \eqref{ineq:Theta}, and Lemma \ref{lem:ACSCS}(d), and following the argument in the proof of Corollary \ref{lm:easyrecur}(a), we conclude that $\min_{1\le j\le k} \phi(x_j)-\phi^*\le \bar \varepsilon$ for $k$ satisfying 
		\[
		k \ge \min \left\lbrace \frac{d_0^2}{\underline{\lam} \bar \varepsilon}, \frac{1+\underline{\lam} \mu}{\underline{\lam} \mu} \log \left( \frac{\mu d_0^2}{\bar \varepsilon} + 1 \right) \right\rbrace,
		\]
        and hence that the statement holds in view of \eqref{ineq:lamj} and the assumption that $\lam_0\ge \bar \varepsilon/(CT_{\bar \varepsilon}^2)$.
    
    b) This statement immediately follows from the update rule in $\lam_j$ and Lemma \ref{lem:ACSCS}(d).
    \end{proof}

    It is worth noting that
    a result similar to
    Corollary \ref{cor:holder} 
    dealing with instances
    $(f,f';h)$ of \eqref{eq:ProbIntro} satisfying (A1), (A2), and \eqref{eq:hybrid-holder}
    can also be established
    for A-CS.
    
	\section{Properties of Bundle Update Schemes (E2) and (E3)}\label{sec:pf E2E3}
	
	This section shows that the update schemes (E2) and (E3) of Subsection \ref{subsec:update} are special implementations of BU.
	
	\begin{proposition}\label{lem:E2}
    Consider the update $\Gamma^+$ of (E2) and set
    $\bar \Gamma = A_f^+ + h$
    where $A_f^+$ is as in \eqref{def:Af+}. Then,
    $(\Gamma^+,\bar \Gamma)$ satisfies \eqref{def:Gamma} and \eqref{def:bar Gamma}.
    As a consequence,
    $\Gamma^+$ is a special implementation of BU.
	\end{proposition}
	
	\begin{proof}
	    First, using the facts that $h \in \mConv{n}$ and $\ell_f(\cdot;z)\le f$ for any $z\in \R^n$, the fact that $\bar \Gamma = A_f^+ + h$, and the definition of $\Gamma^+$ in \eqref{eq:G-agg}, we have 
     \[
        \Gamma^+, \bar \Gamma \in \mConv{n}, \quad \Gamma^+ \le \phi.
     \]
     We have thus shown the inclusion and the second inequality in \eqref{def:Gamma} and the inclusion in \eqref{def:bar Gamma}.
	    It follows from the definitions of $\Gamma^+$ and $\bar \Gamma$ that
	    \[
	        \Gamma^+=\max\{\bar \Gamma, \ell_f(\cdot;x)+h\},
	    \]
	    and hence that $\Gamma^+$ satisfies the first inequality in \eqref{def:Gamma} for any $\tau\in(0,1)$.
	    Moreover, using the fact that $\bar \Gamma = A_f^+ + h$, relation \eqref{theta2}, and the definitions of $\Gamma$ and $A_f^+$ in \eqref{def:Gamma-E2} and \eqref{def:Af+}, respectively, we have
     \[
        \bar \Gamma(x) = A_f^+(x) + h(x) \overset{\eqref{def:Af+}}{=} \theta A_f(x) + (1-\theta) \ell_f(x;x^-) + h(x) \overset{\eqref{theta2}}{=} \max \{A_f(x),\ell_f(x;x^-)\} + h(x) \overset{\eqref{def:Gamma-E2}}{=} \Gamma(x),
     \]
     and hence the first identity in \eqref{def:bar Gamma} holds.
	    Finally, we prove $\bar \Gamma$ satisfies the second identity in \eqref{def:bar Gamma}.
     It follows from the definition of $\ell_f(\cdot;\cdot)$ in \eqref{def:gamma} and relations \eqref{theta1} and \eqref{def:Af+} that
     \[
        \frac1\lam (x-x^c) + \partial h(x) + \nabla A_f^+ \overset{\eqref{def:gamma}, \eqref{def:Af+}}{=}
        \frac1\lam (x-x^c) + \partial h(x) 
+ \theta \nabla A_f + (1-\theta) f'(x^-) \overset{\eqref{theta1}}{\ni} 0.
     \]
	    In conclusion, $\Gamma^+$ as in (E2) is a special way of implementing BU.
	\end{proof}

	\begin{proposition}\label{lem:E3}
    Consider the update $\Gamma^+$ of (E3) and set
    $\bar\Gamma=\Gamma(\cdot;B(x))$
    where $B(x)$ is as in \eqref{def:C+} and  $\Gamma(\cdot;B(x))$ is as in \eqref{eq:Gamma-E2}. Then,
    $(\Gamma^+,\bar \Gamma)$ satisfies \eqref{def:Gamma} and \eqref{def:bar Gamma}.
    As a consequence,
    $\Gamma^+$ is a special implementation of BU.
	\end{proposition}
	
	\begin{proof}
	    First, using the facts that $h \in \mConv{n}$ and $\ell_f(\cdot;b)\le f$ for any $b\in \R^n$,
	    and the definition of $\Gamma(\cdot;B)$ in \eqref{eq:Gamma-E2}, it is easy to see that for any  $B \subset \R^n$, we have
        \begin{equation}\label{eq:Gamma-relation}
            \Gamma(\cdot;B)\in \mConv{n}, \quad \Gamma(\cdot;B) \le \phi.
        \end{equation}
        Recall that
        \begin{equation}\label{eq:twoGammas}
            \Gamma^+=\Gamma(\cdot;B^+), \quad \bar\Gamma=\Gamma(\cdot;B(x)),
        \end{equation}
        hence it follows from \eqref{eq:Gamma-relation} that
     \[
        \Gamma^+, \bar \Gamma \in \mConv{n}, \quad \Gamma^+ \le \phi.
     \]
     We have thus shown the inclusion and the second inequality in \eqref{def:Gamma} and the inclusion in \eqref{def:bar Gamma}.
	    Also, it is easy to see from \eqref{eq:twoGammas}, the first inclusion in \eqref{eq:C+} and the definition of $\Gamma(\cdot;B)$ in \eqref{eq:Gamma-E2} that
	    \[
	    \Gamma^+\ge \max\{\bar \Gamma, \ell_f(\cdot;x)+h\},
	    \]
	    and hence that $\Gamma^+$ satisfies the first inequality in \eqref{def:Gamma} fors any $\tau\in(0,1)$.
	    Moreover, it follows from the fact that $\Gamma=\Gamma(\cdot;B)$, \eqref{eq:Gamma-E2}, \eqref{eq:twoGammas}, and the definition of $B(x)$ in \eqref{def:C+} that the first identity in \eqref{def:bar Gamma} holds.
	    Finally, we prove $\bar \Gamma$ satisfies the second identity in \eqref{def:bar Gamma}.
	    Using
		the definitions of $\Gamma(\cdot;B)$ and $B(x)$ in \eqref{eq:Gamma-E2} and \eqref{def:C+}, respectively,
		and a well-known formula for the subdifferential of the pointwise maximum of finitely many convex functions (e.g., see Corollary 4.3.2 of \cite{urruty1996convex1}), we conclude that 
		\[
		\partial \Gamma(x) = \barco \left(\cup \{ f'(b): b \in B(x)\} \right) + \partial h(x).
		\]
		Using the same reasoning but with $\Gamma$ replaced by $\bar \Gamma$, we conclude that the above set is also $\partial \bar \Gamma(x)$, and hence that
		\[
		\frac{1}{\lam}(x_0-x) \in \partial \Gamma(x)
		=\partial \bar \Gamma (x)
		\]
		where the inclusion is due to \eqref{eq:x-pre}. Now the second identity in \eqref{def:bar Gamma} immediately follows.
		In conclusion, $\Gamma^+$ as in (E3) is a special way of implementing BU.
	\end{proof}

\end{document}